\documentclass[11pt,leqno]{amsart}

\usepackage{amsmath,amsfonts,amssymb}
\usepackage[leqno]{amsmath}
\usepackage{dsfont} 
\usepackage{amsthm} 
\usepackage{enumitem} 
\usepackage[overload]{empheq} 
\usepackage{xfrac}
\usepackage{xcolor}

\theoremstyle{plain} 

\newtheorem{theorem}{Theorem}[section]
\newtheorem{proposition}[theorem]{Proposition}
\newtheorem{corollary}[theorem]{Corollary}
\newtheorem{lemma}[theorem]{Lemma}
\newtheorem{definition}[theorem]{Definition}
\theoremstyle{remark}
\newtheorem{remark}[theorem]{Remark}

\newcommand{\R}{\mathbb{R}}

\newcommand{\N}{\mathbb{N}}

\newcommand{\eps}{\varepsilon}
\newcommand{\ind}{\mathds{1}}
\newcommand{\HH}{\mathcal{H}}
\renewcommand{\b}{\backslash}

\renewcommand{\O}{\Omega}

\newcommand{\q}{\delta}

\def\XXint#1#2#3{{\setbox0=\hbox{$#1{#2#3}{\int}$}
     \vcenter{\hbox{$#2#3$}}\kern-.5\wd0}}

\numberwithin{equation}{section}

\title{Free boundary regularity for a multiphase shape optimization problem}

\author{Luca Spolaor, Baptiste Trey, Bozhidar Velichkov}

\begin{document}

\begin{abstract}
In this paper we prove a $C^{1,\alpha}$ regularity result in dimension two for almost-minimizers of the constrained one-phase Alt-Caffarelli and the two-phase Alt-Caffarelli-Friedman functionals for an energy with variable coefficients. As a consequence, we deduce the complete regularity of solutions of a multiphase shape optimization problem for the first eigenvalue of the Dirichlet-Laplacian up to the fixed boundary. One of the main ingredient is a new application of the epiperimetric-inequality of \cite{spve} up to the boundary. While the framework that leads to this application is valid in every dimension, the epiperimetric inequality is known only in dimension two, thus the restriction on the dimension.
\end{abstract}
\maketitle

\section{Introduction}
In \cite{dt} David and Toro studied properties of the free boundaries for almost-minimizers of the one-phase Alt-Caffarelli and the two-phase Alt-Caffarelli-Friedman functionals considered in \cite{altcaf} and \cite{alcafr} respectively. They proved that, in any dimension, almost-minimizers are non-degenerate and Lipschitz continuous (see also \cite{pippe}). More recently, David-Engelstein-Toro proved in \cite{det} that, under suitable assumption, the free boundaries of almost-minimizers are uniformly rectifiable for both functionals, and almost everywhere given as the graph of a $C^{1,\alpha}$ function for the one-phase functional. 

In this paper we extend these regularity results in dimension two, proving general $C^{1,\alpha}$ regularity results for the free boundary associated to:

{\scriptsize\bf(OP)} {Almost-minimizers} of the one-phase Alt-Caffarelli functional for an operator with variable coefficients, which may also satisfy a further geometric inclusion constraint (Theorem \ref{t:constrained} and Corollary \ref{t:unconstrained});

{\scriptsize\bf(TP)} {Almost-minimizers} of the two-phase Alt-Caffarelli-Friedman functional for an operator with variable coefficients (Theorem \ref{t:two-phase}).

\noindent As pointed out by David-Toro, the difficulty of dealing with almost-minimizers is that they do not satisfy an equation. To overcome this, the approach we follow in this paper is different from the one given in \cite{det} and relies on an epiperimetric inequality (see \cite{spve}) and on a Weiss' almost-monotonicity formula, both of which are variational techniques. We stress out that the only obstruction to a generalization of this proof to any dimension comes from the fact that epiperimetric inequality is only known in dimension two.

The second purpose of this paper is a regularity result for solutions of a multiphase shape optimization problem for the first eigenvalue of the Dirichlet-Laplacian. In \cite{benve} Bogosel and Velichkov proved properties concerning the optimal shapes, such that the lack of triple points, and the Lipschitz continuity of the corresponding eigenfunctions. In this paper we prove, using the above results for almost-minimizers of the one-phase and the two-phase functionals, that if $(\O_1,\dots,\O_n)$ is an optimal shape, then the entire boundary $\partial\O_i$, $i=1,\dots,n$, is $C^{1,\alpha}$ regular (Theorem \ref{t:multi}). This line of study has become increasingly important in recent years, where regularity results for solutions of free-boundary problems, and in particular almost-minimizers, have been applied to study the regularity of shape optimization problems involving eigenvalues of the Dirichlet-Laplacian (see for instance \cite{mateve, KrLi,KrLi2, CaShYe}).

\subsection{Regularity for almost-minimizers}
Throughout this  paper we will use the following notations. 
We fix a matrix valued function $A=(a_{ij})_{ij}: B_2\to Sym_2^+$, where ${Sym}_k^+$ denotes the family of the real positive symmetric $k\times k$ matrices, for which there are constants $\delta,C_{\text{\tiny\sc A}},M_{\text{\tiny\sc A}}>0$ such that
$$|a_{ij}(x)-a_{ij}(y)|\le C_{\text{\tiny\sc A}}|x-y|^\delta,\quad\text{for every}\quad i,j\quad\text{and}\quad  x,y\in B_2\,;$$ 
$$M_{\text{\tiny\sc A}}^{-1}|\xi|^2\le \xi\cdot A(x)\xi=\sum_{i,j=1}^2\xi_i\xi_ja_{ij}(x)\le M_{\text{\tiny\sc A}}|\xi|^2,\quad\text{for every}\quad x\in B_2 \quad\text{and}\quad \xi\in\R^2.$$
We fix $Q_{\text{\tiny\sc op}}$, $Q_{\text{\tiny\sc tp}}^+$ and $Q_{\text{\tiny\sc tp}}^-$ to be three different $\delta$-H\"older continuous functions on $B_2$, for which there is a constant $C_{\text{\tiny\sc Q}}>0$ such that $C_{\text{\tiny\sc Q}}^{-1}\le Q_{\text{\tiny\sc op}}, Q_{\text{\tiny\sc tp}}^+, Q_{\text{\tiny\sc tp}}^-\le C_{\text{\tiny\sc Q}}$ on $B_2$.
Finally, for every function $u:\R^2\to\R$, we will use the following standard notations 
$$u_\pm(x):=\max\{\pm u(x),0\}\,,\quad\Omega_u:=\{u\neq 0\}\,,\quad\Omega_u^+:=\{u> 0\}\quad\text{and}\quad \Omega_u^-:=\{u< 0\}.$$

We are now in position to state our main free boundary regularity results.

\noindent {\bf The one-phase free boundaries.} 
For every $u\in H^1(B_2)$, $x_0 \in B_1$ and $r \in (0,1)$, we define the one-phase functional 
$$J_{\text{\tiny\sc op}}(u,x_0,r)=\int_{B_r(x_0)}\Big(\sum_{i,j}a_{ij}(x)\frac{\partial u}{\partial x_i}\frac{\partial u}{\partial x_j}+Q_{\text{\tiny\sc op}}(x)\ind_{\{u>0\}}\Big)\,dx.$$
Here and after $B_r(x)$ denotes the ball with center $x\in\R^2$ and radius $r>0$ and we will write $B_r:= B_r(0)$. 
Let $\mathcal A^+(B_r)$ be the admissible set 
\begin{align*}\mathcal A^+(B_r)=\big\{ u\in H^1(B_r)\,:\, u\ge0\text{ in }B_r\,,\,\  u=0\text{ on }B_r\setminus B_r^+\big\},\end{align*}
where $H$ stands for the upper half-plane $H:=\big\{(x,y)\in\R^2\,:\,y>0\big\}$ and $B_r^+:=B_r\cap H$.
We say that the nonnegative function $u:B_2\to\R$ is a \emph{almost-minimizer of the one-phase functional $J_{\text{\tiny\sc op}}$ 
in the upper half-disk $B_2^+$}, if $u\in \mathcal A^+(B_2)$ and there are constants $r_1>0$, $C_1>0$ and $\q_1>0$ such that, for every $x_0\in B_1\cap\partial\Omega_u$ and $r\in (0,r_1)$, we have
\begin{align}\label{e:almost1}
 J_{\text{\tiny\sc op}}(u,x_0,r)&\le \big(1+C_1r^{\delta_1}\big)J_{\text{\tiny\sc op}}(v,x_0,r)+C_1r^{2+\q_1}\\
& \quad\text{for every}\ v\in \mathcal A^+(B_2)\ \text{such that}\ u=v\ \text{on}\ B_2\setminus B_r(x_0).\notag
\end{align}
We have the following result for the almost-minimizers of the one-phase Alt-Caffarelli functional $J_{\text{\tiny\sc op}}$ constrained in the upper half-disk $B_2^+$.  
\begin{theorem}[Regularity of the constrained one-phase free boundaries]\label{t:constrained}
Let $B_2\subset\R^2$ and $u:B_2\to\R$ be a non-negative and Lipschitz continuous function. If $u$ is a almost-minimizer of the functional $J_{\text{\tiny\sc op}}$ in $\mathcal A^+(B_2)$, then the free boundary $B_1\cap \partial\Omega_u$ is locally the graph of a $C^{1,\alpha}$ function. 
\end{theorem}
\begin{remark}
The H\"older continuity of the (exterior) normal vector $n_\Omega$ is the best regularity result that one can expect. Indeed, recently Chang-Lara and Savin \cite{savin} showed that even for minimizers the regularity of the constrained free boundaries cannot exceed $C^{1,\sfrac12}$. Moreover, we notice that the result analogous to Theorem \ref{t:constrained} was proved in any dimension in \cite{savin}, by a viscosity approach,  but only for minimizers of the functional $J_{\text{\tiny\sc op}}$. 
\end{remark}
Analogously, we say that the nonnegative function $u:B_2\to\R$ is a \emph{almost-minimizer of the one-phase functional $J_{\text{\tiny\sc op}}$ in $B_2$}, if $u\in H^1(B_2)$ and there are constants $r_1>0$, $C_1>0$ and $\q_1>0$ such that, for every $x_0\in B_1\cap\partial\Omega_u$ and $r\in (0,r_1)$, we have  
\begin{align}\label{e:almost11}
J_{\text{\tiny\sc op}}(u,x_0,r)&\le \big(1+C_1r^{\delta_1}\big)J_{\text{\tiny\sc op}}(v,x_0,r)+C_1r^{2+\q_1}\\
& \quad\text{for every}\ v\in  H^1(B_2)\ \text{such that}\ u=v\ \text{on}\ B_2\setminus B_r(x_0).\notag
\end{align}
The regularity of the unconstrained one-phase free boundary $\partial\Omega_u$ follows directly by Theorem \ref{t:constrained}. For the sake of completeness, we give the precise statement in Corollary \ref{t:unconstrained} below. 
\begin{corollary}[Regularity of the unconstrained one-phase free boundaries]\label{t:unconstrained}
	Let $B_2\subset\R^2$ and $u:B_2\to\R$ be a non-negative and Lipschitz continuous function. If $u$ is a almost-minimizer of the functional $J_{\text{\tiny\sc op}}$ in $B_2$, then the free boundary $B_1\cap \partial\Omega_u$ is locally the graph of a $C^{1,\alpha}$ function. 
\end{corollary}
\begin{remark}
We notice that the regularity of the free boundaries of the one-phase (unconstrained) almost-minimizers is already proved by David, Engelstein and Toro in \cite{det} in every dimension and by a different approach.  
\end{remark}
\noindent{\bf The two-phase free boundaries.} 
For every $u\in H^1(B_2)$, $x_0 \in B_1$ and $r \in (0,1)$, we define the two-phase functional 
$$J_{\text{\tiny\sc tp}}(u,x_0,r)=\int_{B_r(x_0)}\Big(\sum_{i,j}a_{ij}(x)\frac{\partial u}{\partial x_i}\frac{\partial u}{\partial x_j}+Q_{\text{\tiny\sc tp}}^+(x)\ind_{\{u>0\}}+Q_{\text{\tiny\sc tp}}^-(x)\ind_{\{u<0\}}\Big)\,dx.$$
We say that the function $u\in H^1(B_2)$ is a \emph{almost-minimizer of the two-phase functional $J_{\text{\tiny\sc tp}}$ in $B_2$}, if there are constants $r_2>0$, $C_2>0$ and $\delta_2>0$ such that, for every $x_0\in B_1\cap\partial\Omega_u$ and $r\in (0,r_2)$, we have  
\begin{align}\label{e:almost2}
 J_{\text{\tiny\sc tp}}(u,x_0,r)&\le \big(1+C_2r^{\delta_2}\big)J_{\text{\tiny\sc tp}}(v,x_0,r)+C_2r^{2+\q_2}\\
 & \quad\text{for every}\ v\in  H^1(B_2)\ \text{such that}\ u=v\ \text{on}\ B_2\setminus B_r(x_0).\notag
\end{align}
Then, we have the following result: 
\begin{theorem}[Regularity of the two-phase free boundaries]\label{t:two-phase}
Let $B_2\subset\R^2$ and let $u:B_2\to\R$ be Lipschitz continuous such that the functions $u_\pm$ are solutions of the PDEs
$$-\text{\rm div}\,(A\nabla u_\pm)=f_\pm\quad\text{in}\quad\Omega_u^{\pm}\cap B_2,$$
where 
$f_\pm:\overline{\Omega}_u^{\pm}\to\R$ are H\"older continuous functions and $A$ is fixed as above. If $u$ is a almost-minimizer of the two-phase functional $J_{\text{\tiny\sc tp}}$ in $B_2$, then the free boundaries $B_1\cap \partial\Omega_u^+$ and $B_1\cap \partial\Omega_u^-$ are locally graphs of $C^{1,\alpha}$ functions. 
\end{theorem}
\begin{remark}
In particular, we improve from $C^1$ to $C^{1,\alpha}$ the regularity of the free boundaries proved by Alt, Caffarelli and Friedman \cite{alcafr} for minimizers in the case $A=Id$.
\end{remark}

\begin{remark}[Remark on the Lipschitz continuity]
	In Theorem \ref{t:unconstrained} and Theorem \ref{t:two-phase} we assume that the function $u$ is Lipschitz continuous. In the case of the Laplacian, David and Toro \cite{dt} proved that the Lipschitz continuity is a consequence of the the almost-minimality condition. It is of course natural to expect that the same will hold when the operator involved has variable coefficients. We will not address this question in the present paper since our main motivation comes from the application to shape optimization problems as \eqref{e:multi}, for which the Lipschitz continuity is often already known. Actually, in the case of \eqref{e:multi}, the Lipschitz continuity of the eigenfunctions is used to deduce the almost-minimality (see Section \ref{s:multi}).
\end{remark}



\subsection{Multiphase shape optimization problem for the first eigenvalue}
As a consequence of Theorem \ref{t:constrained} and Theorem \ref{t:two-phase}, we get the complete regularity of the following multiphase shape optimization problem 
\begin{equation}\label{e:multi}
\min\Big\{\sum_{i=1}^n\big(\lambda_1(\Omega_i)+q_i|\Omega_i|\big)\ :\ 
\Omega_1,\dots,\Omega_n\ \text{ are disjoint open subsets of }\ \mathcal D\Big\}\,,
\end{equation}
where, we will use the following notations:  

$\quad\bullet$ $1\le n\in\N$, and $0<q_i\in\R$, for every $i=1,\dots,n$;

$\quad\bullet$ $\mathcal D\subset\R^2$ is a bounded open planar set with $C^2$ regular boundary;

$\quad\bullet$ $|\Omega|$ denotes the Lebesgue measure of $\Omega$;

$\quad\bullet$ $\lambda_1(\Omega)$ is the first eigenvalue of the Dirichlet Laplacian on $\Omega$.

\begin{theorem}\label{t:multi}
Let $(\Omega_1,\dots,\Omega_n)$ be a solution of \eqref{e:multi}. Then, all the sets $\Omega_i$, $i=1,\dots,n$, are bounded open sets with $C^{1,\alpha}$ regular boundary. 
\end{theorem}

We notice that, in the above theorem, the {\it entire} boundary $\partial\Omega_i$, $i=1,\dots,n$, is regular. In particular, this holds at the contact points with the other phases $\Omega_j$, $j\neq i$, and also with the boundary of the box $\partial\mathcal D$. 

\begin{remark}[On the regularity of the box]
We state the regularity result of Theorem \ref{t:multi} (and of the corollary below) with a box $D$ of class $C^2$. However, it is possible to weaken this assumption, with exactly the same proof, by assuming that $D$ is a bounded open set of class $C^{1,\alpha}$ such that the solution $w$ to the PDE
\[ -\Delta w = 1 \text{ in } D, \qquad w \in H^1_0(D), \]
is a Lipschitz continuous function in $\R^d$. It is for instance the case if $D$ is $C^{1,1}$ regular.
\end{remark}

Moreover, note that, in the special case $n=1$, \eqref{e:multi} reduces to the classical shape optimization problem 
\begin{equation}\label{eq min lambda meas}
\min \big\{ \lambda_1(\O) + \Lambda |\O| \ : \ \O  \text{ open},\ \Omega\subset \mathcal D\big\}.
\end{equation}
The existence in the class of open sets and the regularity of the free boundary (precisely, of the part contained in the open set $\mathcal D$) was proved by Brian\c con and Lamboley in \cite{brla}. As a direct corollary of our Theorem \ref{t:multi}, we obtain that the entire  boundary is $C^{1,\alpha}$ regular. 

\begin{corollary}[Regularity of the optimal sets for the first eigenvalue]\label{th main th}
Let $\mathcal D \subset \R^2$ be a bounded open set of class $C^2$ and let $\Lambda>0$. Then, there is $\alpha\in(0,1)$ such that every solution $\Omega\subset\mathcal D$ of \eqref{eq min lambda meas} is $C^{1,\alpha}$ regular. 
\end{corollary}

\subsection{Organization of the paper}
In Section \ref{s:epi} we recall the definitions of the Weiss' boundary adjusted energies and the statements of the epiperimetric inequalities. Moreover, we show how to use the algebraic properties of these quantities to deduce the rate of convergence of the blow-up sequences and the uniqueness of the blow-up limits. In Section \ref{s:freezing} we prove a technical lemma that reduces the one-phase and two-phase problems to the case of the Laplacian, which allows us to apply the results of Section \ref{s:epi}. Section \ref{s:blow-up} is dedicated to the classification of the blow-up limits for the one-phase and the two-phase problems. In Section \ref{s:constrained} and Section \ref{s:two-phase} we prove Theorem \ref{t:constrained} and Theorem \ref{t:two-phase}, respectively. In Section \ref{s:multi} we prove that the (eigenfunctions associated to the) solutions of the multiphase problem \eqref{e:multi} are locally almost-minimizers of the one-phase or the two-phase problems, and we prove Theorem \ref{t:multi}.    


\section{Boundary adjusted energy and epiperimetric inequality}\label{s:epi}
All the arguments in this section hold in every dimension $d\ge 2$, except the epiperimetric inequalities Theorem \ref{t:epi1} and Theorem \ref{t:epi2}, which are known only in dimension two. 
\subsection{One-homogeneous rescaling and excess} 
Let $d\ge 2$ and $u \in H^1_{loc}(\R^d)$. For $r>0$ and $x_0 \in \R^d$, we define the one-homogeneous rescaling of $u$ as
\begin{equation}\label{e:resc}
u_{x_0,r}(x):=\frac{u(x_0+rx)}{r}\quad\text{for every}\quad x\in\R^d.
\end{equation}
Then, $u_{x_0,r} \in H^1_{loc}(\R^d)$ and for almost every $r>0$, $E(u_{x_0,r})$ is well defined, where we set
\begin{equation}\label{e:excess}
E(v):=\int_{\partial B_1}|x\cdot \nabla v-v|^2\,d\HH^{d-1},
\end{equation}
where $x\in\partial B_1$ is the exterior normal derivative to $\partial B_1$ at the point $x\in\R^d$ and $\mathcal H^{d-1}$ stands for the $(d-1)$-dimensional Hausdorff measure. The excess function $e(r)=E(u_{x_0,r})$ controls the asymptotic behavior, as $r\to0^+$, of the one parameter family $u_{x_0,r}\in L^2(\partial B_1)$. Precisely, we have the following elementary estimate. 
\begin{lemma}\label{l:excess}
Let $u\in H^1_{loc}(\R^d)$ and $x_0\in\R^d$. Suppose that there are constants $r_0>0$, $\gamma\in(0,1)$ and $I>0$ such that 
\begin{equation}\label{e:intE}
\int_0^{r_0} \frac{E(u_{x_0,r})}{r^{1+\gamma}}\, dr\le I.
\end{equation}
Then, there is a unique function $u_{x_0}\in L^2(\partial B_1)$ such that 
$$\|u_{r,x_0}-u_{x_0}\|_{L^2(\partial B_1)}^2\le \gamma^{-1} I\, r^{\gamma}\quad\text{for every}\quad r\in(0,r_0).$$
\end{lemma}
\begin{proof} 
We set for simplicity, $x_0=0$ and $u_r:=u_{x_0,r}$. Let $0<r<R\le r_0$. 
Notice that, for any $x\in\partial B_1$, we have 
$$\frac{u(Rx)}R-\frac{u(rx)}r=\int_r^R\left(\frac{x\cdot (\nabla u)(sx)}{s}-\frac{u(sx)}{s^2}\right)\,ds=\frac1s\int_r^R\big(x\cdot \nabla u_s(x)-u_s(x)\big)\,ds.$$
Thus, by the Cauchy-Schwarz inequality, we obtain
\begin{align*}
\int_{\partial B_1} |u_R - u_r|^2 \,d\mathcal{H}^{d-1} &\leq \int_{\partial B_1} \left(\int^R_r \frac{1}{s} |x \cdot \nabla u_s - u_s |\,ds \right)^2 \,d\mathcal{H}^{d-1} \\
& \leq \int_{\partial B_1}  \left(\int^R_r s^{\gamma-1}ds\right) \left( \int^R_r \frac{1}{s^{1+\gamma}} |x \cdot \nabla u_s - u_s |^2 ds \right) \,d\mathcal{H}^{d-1} \\
&\leq \frac{R^\gamma-r^\gamma}{\gamma}\int^R_r \frac{E(u_s)}{s^{1+\gamma}}\, ds\leq \frac{R^\gamma}{\gamma}\int^{r_0}_0 \frac{E(u_s)}{s^{1+\gamma}}\, ds,
\end{align*}
which implies the claim by a standard argument. \end{proof}
\subsection{The one-phase boundary adjusted energy}
Let $d\ge 2$ and $u\in H^1(B_1)$. For any $\Lambda> 0$, we define 
the one-phase Weiss' boundary adjusted energy as 
\begin{equation}\label{e:weiss1}
W_{\text{\sc op}}(u) := \int_{B_1} |\nabla u|^2 \,dx - \int_{\partial B_1}u^2\,d\mathcal{H}^{d-1} + \Lambda \big|{\{u>0\}} \cap B_1\big|.
\end{equation}
Let $r>0$, $x_0 \in \R^2$ and $u \in H^1_{loc}(\R^d)$. The  relation between $W_{\text{\sc op}}$ and the excess $E$ is given by the following formula, which holds for any function $u$ and can be obtained by a direct computation (see \cite{weiss} and \cite{mateve}).
\begin{equation}\label{e:der1}
\frac{\partial}{\partial r}W_{\text{\sc op}}(u_{x_0,r}) = \frac{d}{r}\big(W_{\text{\sc op}}(z_{x_0,r})-W_{\text{\sc op}}(u_{x_0,r})\big)+\frac{1}{r}E(u_{x_0,r}),
\end{equation}
where $z_{x_0,r}$ denotes the one-homogeneous extension of the trace of $u_{x_0,r}$ in $B_1$, that is, 
\begin{equation}\label{e:zeta}
z_{x_0,r}(x):=|x|\,u_{x_0,r}\left(\sfrac{x}{|x|}\right) = \frac{|x|}{r}\,u\left(\sfrac{rx}{|x|}\right)\,, \quad \text{for every}\quad x \in B_1. 
\end{equation}

\begin{theorem}[Epiperimetric inequality for $W_{\text{\sc op}}$]\label{t:epi1}
Let $d=2$. Let $C_0>0$ be a given constant. There exists a constant $\varepsilon>0$ such that: for every non-negative $c\in H^1(\partial B_1)$ satisfying $\displaystyle\int_{\partial B_1}c\,d\HH^1\ge C_0$, there exists a non-negative function $h \in H^1(B_1)$ such that $h=c$ on $\partial B_1$ and 
\begin{equation}\label{eq epi}
W_{\text{\sc op}}(h) - \Lambda \frac{\pi}{2} \leq (1-\varepsilon)\left(W_{\text{\sc op}}(z) - \Lambda \frac{\pi}{2}\right),
\end{equation}
where $W_{\text{\sc op}}$ is given by \eqref{e:weiss1} and $z\in H^1(B_1)$ denotes the one-homogeneous extension of $c$ into $B_1$. Moreover, the competitor $h$ has the following properties:
\begin{enumerate}
\item[(a)] There is a universal numerical constant $C>0$ such that $\|h\|_{H^1(B_1)}\le C\|c\|_{H^1(\partial B_1)}$.
\item[(b)] If $H_{x_0,\nu}:=\big\{x\in\R^2\,:\,(x-x_0)\cdot\nu \geq 0\big\}$, for some $x_0\in\R^2$ and $\nu\in\partial B_1$, is a half-plane such that 
\begin{equation}\label{e:halfplane}
0\in\mathcal \overline{H}_{x_0,\nu}\qquad\text{and}\qquad z=0\quad\text{on}\quad \R^2\setminus H_{x_0,\nu}\,,
\end{equation}
then we can choose $h$ such that $h=0$ on  $\R^2\setminus H_{x_0,\nu}$. 
\end{enumerate}
\end{theorem}

\subsection{The two-phase boundary adjusted energy}
For every $\Lambda_1,\Lambda_2>0$ and $v\in H^1(B_1)$, we define the two-phase Weiss' boundary adjusted energy  as
\begin{equation}\label{e:weiss2}
W_{\text{\sc tp}}(v) =\int_{B_1} |\nabla v|^2 \,dx - \int_{\partial B_1}v^2\,d\mathcal{H}^{d-1} +\Lambda_1|{\{v>0\}} \cap B_1| + \Lambda_2|{\{v<0\}} \cap B_1|.
\end{equation}
As in the one phase case, we have 
\begin{equation}\label{e:der2}
\frac{\partial}{\partial r}W_{\text{\sc tp}}(u_{x_0,r}) = \frac{d}{r}\big(W_{\text{\sc tp}}(z_{x_0,r})-W_{\text{\sc tp}}(u_{x_0,r})\big)+\frac{1}{r}E(u_{x_0,r}),
\end{equation}
where $z_{x_0,r}$ is given by \eqref{e:zeta}.

\begin{theorem}[Epiperimetric inequality for $W_{\text{\sc tp}}$]\label{t:epi2}
Let $d=2$. For every $C_0>0$ there is $\varepsilon>0$ such that: for every $c \in H^1(\partial B_1)$ such that $\displaystyle\int_{\partial B_1} c^+ \,d\mathcal{H}^1 \geq C_0$ and $\displaystyle\int_{\partial B_1} c^- \,d\mathcal{H}^1 \geq C_0$, there exists a function $h \in H^1(B_1)$ with $h=c$ on $\partial B_1$ such that
\begin{equation}\label{eq epi DP}
W_{\text{\sc tp}}(h) - (\Lambda_1 + \Lambda_2) \frac{\pi}{2} \leq (1-\varepsilon)\left(W_{\text{\sc tp}}(z) - (\Lambda_1 + \Lambda_2) \frac{\pi}{2}\right),
\end{equation}
where $z\in H^1(B_1)$ is the one-homogeneous extension of the trace of $c$ to $B_1$. Moreover, there is a universal numerical constant $C>0$ such that $\|h\|_{H^1(B_1)}\le C\|c\|_{H^1(\partial B_1)}$.
\end{theorem}

\subsection{Almost-monotonicity and almost-minimality} 
Let $u\in H^1_{loc}(\R^d)$ and $x_0\in\R^d$. For any $r>0$, the function $u_{x_0,r}$ and $z_{x_0,r}$ are defined as in \eqref{e:resc} and \eqref{e:zeta}, respectively. In the next lemma we will show that a almost-minimality of $u$, with respect to radial perturbations, implies that the function $r\mapsto W_{\square}(u_{x_,r})$ is monotone up to a small error term ($\square$ stands for {\sc op} or {\sc tp}). 
\begin{lemma}[Monotonicity of $W_\square$]\label{l:mono}
Let $u\in H^1_{loc}(\R^d)$ and $x_0\in\R^d$. Suppose that there are constants $r_0>0$, $C>0$ and $\q>0$ such that 
\begin{equation}\label{e:almost_zeta}
W_{\square}(u_{x_0,r})\le W_\square(z_{x_0,r})+Cr^\q\quad\text{for every}\quad r\in(0,r_0),
\end{equation}
where $\square$ stands for {\sc op} or {\sc tp}. Then, the function 
\begin{equation}\label{e:mono}
r\mapsto W_\square(u_{x_0,r})+\frac{Cd}{\q}r^\q,
\end{equation}
is non-decreasing on the interval $(0,r_0)$. 
\end{lemma}
\begin{proof}
Using \eqref{e:der1} for $\square=${\sc op} (resp. \eqref{e:der2} for $\square=${\sc tp}), and the condition \eqref{e:almost_zeta} we get  
\begin{align*}
\frac{\partial }{\partial r}W_\square(u_{x_0,r})&\ge \frac{d}{r}\left(W_\square(z_{x_0,r})-W_\square(u_{x_0,r})\right)\ge Cd\, r^{\q-1},
\end{align*}
which gives \eqref{e:mono}.
\end{proof}

\subsection{Epiperimetric inequality and energy decay}
In this section we show how to use the epiperimetric inequality to obtain at once the decay for the energy $W_\square(u_{x_0,r})$ and the convergence of $u_{x_0,r}$ in $L^2(\partial B_1)$. The argument is very general and we treat the cases $\square=\text{\sc op}$ and $\square=\text{\sc tp}$ simultaneously.
\begin{lemma}\label{l:decay}
Let $u\in H^1_{loc}(\R^d)$, $x_0\in\R^d$ and $W_\square$ be as in \eqref{e:weiss1}, if $\square=\text{\sc op}$, and \eqref{e:weiss2}, if $\square=\text{\sc tp}$. Suppose that there are constants $r_0\in(0,1)$, $C>0$, $\q>0$ and $\varepsilon\in(0,\frac{\delta}{2d+\delta})$ such that: 

\noindent {\rm(a)} \eqref{e:almost_zeta} holds and the limit $\displaystyle\Theta_\square:=\lim_{r\to0}W_\square(u_{x_0,r})$ (which exists due to Lemma \ref{l:mono}) is finite; 

\noindent {\rm(b)} for every $r\in(0,r_0)$ there is a function $h_{x_0,r}\in H^1(B_1)$ such that  
\begin{equation}\label{e:almost_h}
W_{\square}(u_{x_0,r})\le 
W_\square(h_{x_0,r})+Cr^\q,
\end{equation}
and we have the epiperimetric inequality
\begin{equation}\label{e:epi_square}
W_{\square}(h_{x_0,r})-\Theta_{\square}\le (1-\varepsilon)\big(W_{\square}(z_{x_0,r})-\Theta_{\square}\big).
\end{equation}

Then, there is a unique function $u_{x_0}\in L^2(\partial B_1)$ such that 
$$\|u_{r,x_0}-u_{x_0}\|_{L^2(\partial B_1)}^2\le \gamma^{-1} I\, r^{\gamma}\quad\text{for every}\quad r\in(0,r_0),$$
where $\gamma=\frac{d\eps}{1-\eps}$ and $\displaystyle  I=r_0^{-\gamma}\big(W_\square(u_{x_0,r_0})-\Theta_\square\big)+\frac{dC}{\q-\gamma}r_0^{\q-\gamma}$.
\end{lemma}
\begin{proof}
We use \eqref{e:der1} for $\square=${\sc op} (resp. \eqref{e:der2} for $\square=${\sc tp}), then the epiperimetric inequality \eqref{e:epi_square} and the almost-minimality condition \eqref{e:almost_h}. 
\begin{align*}
\frac{\partial }{\partial r}\big(W_\square(u_{x_0,r})-\Theta_\square\big)&\ge \frac{d}{r}\Big(\big(W_\square(z_{x_0,r})-\Theta_\square\big)-\big(W_\square(u_{x_0,r})-\Theta_\square\big)\Big)\\
&\ge \frac{d}{r}\Big(\frac1{1-\eps}\big(W_\square(h_{x_0,r})-\Theta_\square\big)-\big(W_\square(u_{x_0,r})-\Theta_\square\big)\Big)\\
&\ge \frac{d}{r}\Big(\frac\eps{1-\eps}\big(W_\square(u_{x_0,r})-\Theta_\square\big)-Cr^\q\Big),
\end{align*}
which implies that the function 
$$f(r)=\frac{W_\square(u_{x_0,r})-\Theta_\square}{r^{\gamma}}+\frac{dC}{\q-\gamma}r^{\q-\gamma}$$
is non-decreasing on $(0,r_0)$ for $\gamma=\frac{d\eps}{1-\eps}$, where we notice that $\gamma\le\frac\delta2$ due to the choice $\eps\le \frac{\q}{2d+\q}$. In particular, using again \eqref{e:der1} (resp. \eqref{e:der2}), we get 
\begin{align*}
f'(r)\ge\frac{1}{r^{\gamma+1}}E(u_{x_0,r}),
\end{align*}
which integrated gives
$$f(r_0)-f(s)\ge \int_s^{r_0}\frac{1}{r^{\gamma+1}}E(u_{x_0,r})\,dr,$$
for every $s\in(0,r_0)$. 
Now, notice that, up to choosing a bigger constant $C$ in \eqref{e:almost_h}, Lemma \ref{l:mono} implies that $f(s)\ge 0$ for every $s>0$. Thus, we get 
$$f(r_0)\ge \int_0^{r_0}\frac{1}{r^{\gamma+1}}E(u_{x_0,r})\,dr,$$
which is precisely \eqref{e:intE} with $I:=f(r_0)$.
\end{proof}	

\section{Change of variables and freezing of the coefficients}\label{s:freezing}
The arguments of the previous section, the monotonicity formula and the decay of the blow-up sequences, can be applied only in the case when the operator in $J_{\text{\tiny\sc op}}$ (resp. $J_{\text{\tiny\sc tp}}$) is the identity. Thus, in order to prove the regularity results Theorem \ref{t:constrained} and Theorem \ref{t:two-phase} we need to change the coordinates and reduce to the case $A=Id$. We prove the main estimate of this section in Lemma \ref{l:change_of_variables} below, but before we will introduce several notations.
\smallskip 

Let $A=(a_{ij})_{ij}:B_2\to Sym_2^+$ and $Q_{\text{\tiny\sc op}},Q_{\text{\tiny\sc tp}}^+,Q_{\text{\tiny\sc tp}}^-:B_2\to\R^+$ be as in the Introduction and note that we have 
$$\|A_x^{\sfrac12}\|\le M_{\text{\tiny \sc A}}^{\sfrac12}\quad\text{and}\quad \|A_x^{-\sfrac12}\|\le M_{\text{\tiny \sc A}}^{\sfrac12}\quad\text{for every}\quad x\in B_2,$$ 
where $\|A\|=\sup\big\{|Au|\,:\,u\in\R^2,\ |u|=1\big\}$.

\begin{remark}
We recall that if $M\in Sym_d^+$, then there is an orthogonal matrix $P$ such that $PMP^t=\text{diag}(\lambda_1,\dots,\lambda_d)$, where $P^t$ is the transpose of $P$ and $\text{diag}(\lambda_1,\dots,\lambda_d)$ is the diagonal matrix with eigenvalues $\lambda_1,\dots,\lambda_d$. We set $D=\text{diag}(\sqrt\lambda_1,\dots,\sqrt\lambda_d)$ and define $M^{\sfrac12}:=P^tDP$.
\end{remark}

For $x_0\in B_2$ and $r>0$ we define the functionals
$$J_{\text{\tiny\sc op}}^{x_0}(v,r):=\int_{B_r} \Big(|\nabla v|^2 +Q_{\text{\tiny\sc op}}(x_0)\ind_{\{v>0\}}\Big)\,dx\,;$$
$$J_{\text{\tiny\sc tp}}^{x_0}(v,r):=\int_{B_r} \Big(|\nabla v|^2 +Q_{\text{\tiny\sc tp}}^+(x_0)\ind_{\{v>0\}}+Q_{\text{\tiny\sc tp}}^-(x_0)\ind_{\{v<0\}}\Big)\,dx\,.$$
For every $x_0\in B_1$, we define the function 
\begin{equation}\label{e:F_x_0}
F_{x_0}(x):=x_0+A_{x_0}^{\sfrac12}(x)
\end{equation}
and the half-plane $\mathcal H_{x_0}:=\big\{x\in\R^2\, :\, F_{x_0}(x)\cdot e_2>0\big\}\,,\ \text{\,where\,}\ e_2=(0,1).$

\begin{lemma}\label{l:change_of_variables}
Let $L>0$. There are constants $C>0$ and $r_0\in(0,1)$ (depending only on $C_{\text{\tiny \sc A}}$, $C_{\text{\tiny \sc Q}}$, $M_{\text{\tiny \sc A}}$, $M_{\text{\tiny \sc Q}}, \delta_{\text{\tiny \sc A}},\delta_{\text{\tiny \sc Q}},\delta_1,C_1$ and $L$) and $\delta=\min\{\delta_{\text{\tiny \sc A}}, \delta_{\text{\tiny \sc Q}},\delta_1\}$ such that: if $u\in H^1(B_1)$ is a nonnegative $L$-Lipschitz continuous function and a almost-minimizer of $J_{\text{\tiny\sc op}}$ in $B_2^+$, $x_0\in B_{r_0}\cap\partial \Omega_u$ and $\bar u=u\circ F_{x_0}$ ($F_{x_0}$ is defined in \eqref{e:F_x_0} above), then  for every $r\in(0,r_0)$, 
\begin{align}\label{e:almost-minimality_in_x_0}
J_{\text{\tiny\sc op}}^{x_0}(\bar u, r)&\le (1+Cr^\delta)J_{\text{\tiny\sc op}}^{x_0}(\bar v, r)+C r^{2+\delta},\\
&\text{for every $\bar v\in H^1(B_r)$ such that $\bar u-\bar v\in H^1_0(B_r)$ and $\bar v=0$ on $\R^2\setminus \mathcal H_{x_0}$.}\notag
\end{align} 
Moreover, there is a numerical constant $C_0>0$, such that
$$W_{\text{\tiny\sc op}}(\bar u_r)\le \begin{cases} W_{\text{\tiny\sc op}}(\bar z_r)+C_0(M_{\text{\tiny\sc A}}L^2+M_Q)C r^\delta,\\
W_{\text{\tiny\sc op}}(\bar h_r)+C_0(M_{\text{\tiny\sc A}}L^2+M_Q)C r^\delta,
\end{cases}$$
for every $r\in(0,r_0)$, where $C$ is the constant from \eqref{e:almost-minimality_in_x_0}, $\bar u_r(x):=\frac1r \bar u(rx)$, $\bar z_r$ is the one homogeneous extension of $\bar u_r$ in $B_1$, $\bar h_r$ is the competitor given by Theorem \ref{t:epi1} and $\Lambda = Q_{\text{\tiny\sc op}}(x_0)$ in \eqref{e:weiss1}. 
\end{lemma}
\begin{proof}
Let $x_0\in\partial \Omega_u\cap B_1$, $r>0$ and $\rho=M_{\text{\tiny\sc A}}^{\sfrac12}r$ and notice that this implies $F_{x_0}(B_r)\subset B_\rho(x_0)$. Let $\bar u=u \circ F_{x_0}$ and $\bar v=v\circ F_{x_0}$. Then, the H\"older continuity of $A$ and $Q:=Q_{\text{\tiny\sc op}}$ and the ellipticity of $A$ give
\begin{align*}
\tilde{J}_{\text{\tiny\sc op}}(u,x_0,\rho) &:= \int_{B_\rho(x_0)}\Big(a_{ij}(x_0)\,\partial_i u\,\partial_j u + Q(x_0)\ind_{\{u>0\}}\Big)\,dx \leq J_{\text{\tiny\sc op}}(u,x_0,\rho) \\
& \qquad + C_{\text{\tiny\sc A}}M_{\text{\tiny\sc A}}\rho^{\delta_{\text{\tiny\sc A}}} \int_{B_\rho(x_0)}a_{ij}(x)\,\partial_i u\,\partial_j u\,dx + C_{\text{\tiny\sc Q}}M_{\text{\tiny\sc Q}}\rho^{\delta_{\text{\tiny\sc Q}}}\int_{B_\rho(x_0)}Q(x)\ind_{\{u>0\}}\,dx\\ 
&\leq (1+Cr^\delta) J_{\text{\tiny\sc op}}(u,x_0,\rho),
\end{align*}
for some positive constant $C>0$.
Analogously, we get the following estimate from below:
\begin{equation}\label{e:esti J coeff cst}
\tilde{J}_{\text{\tiny\sc op}}(v,x_0,\rho) \geq (1-Cr^\delta) J_{\text{\tiny\sc op}}(v,x_0,\rho).
\end{equation}
Putting the two estimates together and using the almost-minimality of $u$, we get 
\[
\tilde{J}_{\text{\tiny\sc op}}(u,x_0,\rho) \leq 
\frac{1+Cr^\delta}{1-Cr^\delta}(1+C_1\rho^{\delta_1}) \tilde{J}_{\text{\tiny\sc op}}(v,x_0,\rho) + C_1(1+Cr^\delta)\rho^{2+\delta_1}.
\]
Now, notice that by the choice of the function $F_{x_0}$ we have the identity
\[
|\nabla \bar{u}|^2(x) = a_{ij}(x_0)\,\partial_i u(F_{x_0}(x))\,\partial_j u(F_{x_0}(x)), \qquad x \in B_{\text{\tiny\sc M}_{\text{\tiny\sc A}}^{-\sfrac12}}.
\]
Therefore, a change of coordinates and the estimate \eqref{e:esti J coeff cst} give
\begin{align*}
&\int_{F_{x_0}^{-1}(B_\rho(x_0))}\Big(|\nabla\bar{u}|^2+Q(x_0)\ind_{\{\bar{u}>0\}}\Big)\,dx = \det \big(A_{x_0}^{-\sfrac12}\big) \tilde{J}_{\text{\tiny\sc op}}(u,x_0,\rho) \\
&\leq (1+Cr^\delta)\int_{F_{x_0}^{-1}(B_\rho(x_0))}\Big(|\nabla\bar{v}|^2+Q(x_0)\ind_{\{\bar{v}>0\}}\Big)\,dx +Cr^{2+\delta},
\end{align*}
for some other positive constant $C>0$. Finally, observing that $B_r \subset F_{x_0}^{-1}(B_\rho(x_0))$ we get
\[ J_{\text{\tiny\sc op}}^{x_0}(\bar u,r) \leq (1+Cr^\delta)J_{\text{\tiny\sc op}}^{x_0}(\bar v,r) + Cr^{2+\delta} + Cr^\delta J_{\text{\tiny\sc op}}^{x_0}(\bar u, \text{\tiny\sc M}_{\text{\tiny\sc A}}^{-\sfrac12} r),
\]
which gives \eqref{e:almost-minimality_in_x_0} since $\bar{u}$ is Lipschitz continuous with Lipschitz constant $\|\nabla \bar{u}\|_{L^\infty}=M_{\text{\tiny \sc A}}^{\sfrac12}L$.

We next notice that we have the scaling
$$J_{\text{\tiny\sc op}}^{x_0}(\bar u_r, 1)=\frac1{r^2}J_{\text{\tiny\sc op}}^{x_0}(\bar u, r).$$
Thus, the almost-minimality inequality \eqref{e:almost-minimality_in_x_0} translates in  
\begin{equation}\label{e:sto_gia_dormendo}
J_{\text{\tiny\sc op}}^{x_0}(\bar u_r, 1)\le (1+Cr^\delta)J_{\text{\tiny\sc op}}^{x_0}(\bar v_r, 1)+C r^{\delta}.
\end{equation} 
Let $C_{\text{\tiny\sc E}}>0$ be the constant from Theorem \ref{t:epi1}. Then, since $\bar u$ is Lipschitz continuous, we have 
$$\displaystyle \int_{B_1}|\nabla \bar h_r|^2\,dx\le C_{\text{\tiny\sc E}}\int_{\partial B_1}\big(|\nabla \bar u_r|^2+\bar u_r^2\big)\,dx\le C_0M_{\text{\tiny\sc A}}L^2,$$ 
where $C_0$ is a numerical constant and $\bar h_r$ is the competitor from Theorem \ref{t:epi1}. Taking $\bar h_r$ as a competitor in \eqref{e:sto_gia_dormendo}, we obtain 
\begin{align*}
J_{\text{\tiny\sc op}}^{x_0}(\bar u_r, 1)&\le J_{\text{\tiny\sc op}}^{x_0}(\bar h_r, 1)+Cr^\delta \Big(\int_{B_1}|\nabla \bar h_r|^2\,dx +Q(x_0)|B_1|\Big)+C r^{\delta}\\
&\le J_{\text{\tiny\sc op}}^{x_0}(\bar h_r, 1)+Cr^\delta \Big(C_0 M_{\text{\tiny\sc A}} L^2 + M_{\text{\tiny\sc Q}}|B_1|\Big)+C r^{\delta},
\end{align*} 
which concludes the proof, the case $\bar v_r=\bar z_r$ being analogous.
\end{proof} 
An analogous result, with essentially the same proof holds in the two-phase case. 
\begin{lemma}\label{l:change_of_variables2}
Let $L>0$. There are constants $C>0$ and $r_0\in(0,1)$ (depending only on $C_{\text{\tiny \sc A}}$, $C_{\text{\tiny \sc Q}}$, $M_{\text{\tiny \sc A}}$, $M_{\text{\tiny \sc Q}}, \delta_{\text{\tiny \sc A}},\delta_{\text{\tiny \sc Q}},\delta_1,C_2$ and $L$) and $\delta=\min\{\delta_{\text{\tiny \sc A}}, \delta_{\text{\tiny \sc Q}},\delta_2\}$ such that: if $u\in H^1(B_1)$ is a $L$-Lipschitz continuous function and a almost-minimizer of $J_{\text{\tiny\sc tp}}$ in $B_2$, $x_0\in B_{r_0}\cap\partial \Omega_u$ and $\bar u=u\circ F_{x_0}$, then we have that for every $r\in(0,r_0)$, 
\begin{align}\label{e:almost-minimality_in_x_0_due}
J_{\text{\tiny\sc tp}}^{x_0}(\bar u, r)&\le (1+Cr^\delta)J_{\text{\tiny\sc tp}}^{x_0}(\bar v, r)+C r^{2+\delta},\\
&\text{for every $\bar v\in H^1(B_r)$ such that $\bar u-\bar v\in H^1_0(B_r)$.}\notag
\end{align} 
Moreover, there is a numerical constant $C_0>0$ such that
$$W_{\text{\tiny\sc tp}}(\bar u_r)\le \begin{cases} W_{\text{\tiny\sc tp}}(\bar z_r)+C_0(M_{\text{\tiny \sc A}} L^2+M_Q)C r^\delta,\\
W_{\text{\tiny\sc tp}}(\bar h_r)+C_0(M_{\text{\tiny \sc A}} L^2+M_Q)C r^\delta,
\end{cases}$$
 for every $r\in(0,r_0)$, where $C$ is the constant from \eqref{e:almost-minimality_in_x_0_due}, $\bar u_r(x):=\frac1r \bar u(rx)$, $\bar z_r$ is the one homogeneous extension of $\bar u_r$ in $B_1$, $\bar h_r$ is the competitor given by Theorem \ref{t:epi2} and $\Lambda_1 = Q_{\text{\tiny\sc tp}}^+(x_0)$, $\Lambda_2 = Q_{\text{\tiny\sc tp}}^-(x_0)$ in \eqref{e:weiss2}. 
\end{lemma}

\begin{remark}[On the non-degeneracy]\label{rk nondeg}
In \cite{dt} David and Toro proved that Lipschitz continuous almost-minimizers to the one-phase and the two-phase functionals for the Laplacian are non degenerate (see \cite[Theorem 10.1]{dt}). Note that their definition of almost-minimizer is slightly different from ours. However, their proof still holds in our case with small changes which come from the additional term $Cr^{2+\delta}$ of our definition. It follows from Lemma \ref{l:change_of_variables} and Lemma \ref{l:change_of_variables2} that if $u$ is a almost-minimizer of the functional $J_{\text{\tiny\sc op}}$ (resp. $u$ is a almost-minmizer of $J_{\text{\tiny\sc tp}}$) then $u$ (resp. $u_\pm$) is non-degenerate with respect to $A$ in the sense of the following definition.
\end{remark}

\begin{definition}[Non-degeneracy]\label{d:non-degeneracy}
	Let $d\ge 2$ and $A:\R^d\to Sym_d^+$ be a given function.
	We say that the non-negative function $u\in H^1(B_2)$ is non-degenerate (with respect to $A$), if there are constants $\eta>0$, $\eps\in(0,1)$ and $r_0>0$ such that, for every $x_0\in B_1$ and $r\in(0,r_0)$, the following implication holds:
	$$\int_{\partial B_r}u\circ F_{x_0}\,d\HH^{d-1}< \eta\, r^d\qquad\Rightarrow\qquad u\circ F_{x_0}\equiv 0\quad\text{in}\quad B_{\eps r}(x_0),$$
	where $F_{x_0}(x):=x_0+A_{x_0}^{\sfrac12}(x)$.
\end{definition}

\section{Blow-up sequences and blow-up limits}\label{s:blow-up}
Let $u\in H^1(B_2)$ be a Lipschitz continuous function. Let $(x_n)_{n\ge 1}$ be a sequence of points in $B_1\cap\partial\Omega_u$ converging to some $ x_0\in B_1\cap\partial\Omega_u$, and $(r_n)_{n\ge 1}$ be an infinitesimal sequence in $(0,1)$. 
Then, the sequence $u_{x_n,r_n}$ is uniformly Lipschitz in every compact subset of $\R^2$. Thus, up to extracting a subsequence, there is a Lipschitz function $u_0:\R^2\to\R$ such that 
\begin{equation}\label{e:uniform}
\lim_{n\rightarrow \infty}u_{x_n,r_n} = u_{0},
\end{equation}
where $u_{r_n,x_n}$ is defined in \eqref{e:resc} and the convergence is uniform on every compact subset of $\R^2$. 
\begin{definition}
If \eqref{e:uniform} holds, we will say that $u_{x_n,r_n}$ is a blow up sequence (with fixed center, if $x_n=x_0$, for every $n\ge 1$). If the center is fixed, we will say that $u_0$ is a blow-up limit at $x_0$.  
\end{definition}
We summarize the main properties of the blow-up sequences and the blow-up limits in the following two propositions. We notice that Proposition \ref{p:convergence} holds in every dimension $d\ge 2$, while Proposition \ref{p:classification} is known to hold only for $2\le d\le 4$. 
\begin{proposition}[Convergence of the blow-up sequences]\label{p:convergence}
Let $u \in H^1(B_2)$ be as in Theorem \ref{t:constrained} or Theorem \ref{t:two-phase} and let $u_n:=u_{r_n,x_n}$ be a blow-up sequence converging to some $u_0\in H^1_{loc}(\R^2)$. Then: 
\begin{enumerate}
	\item the sequence $u_{n}$ converges strongly to $u_{0}$ in $H^1_{\text{loc}}(\R^2)$;
	\item the sequences of characteristic functions $\ind_{\{u_n>0\}}$ and $\ind_{\{u_n<0\}}$ converge  in $L^1_{\text{loc}}(\R^2)$ to $\ind_{\{u_{0}>0\}}$ and $\ind_{\{u_{0}<0\}}$, respectively.
\end{enumerate}
\end{proposition}

\begin{proposition}[Classification of the blow-up limits]\label{p:classification}
Let $u \in H^1(B_2)$ be as in Theorem \ref{t:constrained} or Theorem \ref{t:two-phase}. Let  $x_0\in\partial\Omega_u\cap B_1$ and $u_0\in H^1_{loc}(\R^2)$ be a blow-up limit of $u$ at $x_0$.
\smallskip
\begin{enumerate}
\item[{\bf \small (OP)}]  If $u$ is as in Theorem \ref{t:constrained} and $x_0\in \partial \Omega_u\cap B_1^+$, then $u_0$ is of the form 
\begin{equation}\label{e:blowup1}
u_0(x)=Q_{\text{\tiny\sc op}}^{\sfrac12}(x_0)\max\big\{0,x\cdot A_{x_0}^{-\sfrac12}[\nu]\big\}\,,\quad\text{where}\quad \nu\in\partial B_1.
\end{equation}
\item[{\bf \small (OP-c)}]  If $u$ is as in Theorem \ref{t:constrained} and $x_0\in \partial \Omega_u\cap\partial H\cap  B_1$, then $u_0$ is of the form 
\begin{equation}\label{e:blowup1c}
u_0(x)=\mu \max\big\{0,x\cdot A_{x_0}^{-\sfrac12}[\nu]\big\},
\end{equation}
where $\mu\ge Q_{\text{\tiny\sc op}}^{\sfrac12}(x_0)$ and $\nu\in\partial B_1$ is such that $A_{x_0}^{-\sfrac12}[\nu]$ is normal to $\partial H$ and pointing inwards. 
\item[{\bf \small (TP)}]  If $u$ is as in Theorem \ref{t:two-phase} and $x_0\in \partial \Omega_u^+\cap \partial \Omega_u^-\cap B_1$, then $u_0$ is of the form 
\begin{equation}\label{e:blowup2}
u_0(x)=\mu_+\max\big\{0,x\cdot A_{x_0}^{-\sfrac12}[\nu]\big\}+\mu_-\min\big\{0,x\cdot A_{x_0}^{-\sfrac12}[\nu]\big\},
\end{equation}
for some $\nu\in\partial B_1$ and some $\mu_+,\mu_->0$ such that $\mu_+^2-\mu_-^2=Q_{\text{\tiny\sc tp}}^+(x_0)-Q_{\text{\tiny\sc tp}}^-(x_0)$ and $\mu_+^2\geq Q_{\text{\tiny\sc tp}}^+(x_0)$, $\mu_-^2\geq Q_{\text{\tiny\sc tp}}^-(x_0)$.
\end{enumerate}
\end{proposition}		

The proof of Proposition \ref{p:convergence} follows by a standard variational argument that only uses the almost-minimality of $u$; for more details, we refer to \cite{altcaf} (see also \cite{mateve}). Proposition \ref{p:classification} follows by  the optimality of the blow-up limits and the Weiss' monotonicity formula (Lemma \ref{l:mono}). 
We will need the following definition.	
\begin{definition}[Global solutions]
Let $u:\R^2\to\R$, $u\in H^1_{loc}(\R^2)$ be given. 
\begin{enumerate}
\item[{\bf \small (OP)}] We say that $u$ is a global solution of the {\bf one-phase} Bernoulli problem, if: $u\ge 0$ and, for every ball $B:=B_R(x_0)\subset\R^2$, we have 
\begin{equation}\label{e:global1}\int_{B} |\nabla u|^2 \,dx + \Lambda|\{u>0\} \cap B| \leq \int_{B}|\nabla v|^2 \,dx + \Lambda |\{v>0\} \cap B|, 
\end{equation}
for every $v \in H^1(B)$ such that $u-v \in H^1_0(B)$.
\item[{\bf \small (OP-c)}] We say that 	$u$ is a global solution of the {\bf one-phase constrained} Bernoulli problem in $H$, if $u\ge 0$ on $H$, $u=0$ on $\R^2\setminus H$ and \eqref{e:global1} holds, for every ball $B:=B_R(x_0)\subset\R^2$ and every $v \in H^1(B)$ such that $u-v \in H^1_0(B)$ and $\{v>0\} \subset H$.
\item[{\bf \small (TP)}] We say that $u$ is a global solution of the {\bf two-phase} Bernoulli problem if, for every ball $B:=B_R(x_0)\subset\R^2$, we have 
\begin{equation}\label{e:global2}
\int_{B} \Big(|\nabla u|^2+\Lambda_1\ind_{\{u>0\}}+\Lambda_2\ind_{\{u<0\}}\Big)\,dx
 \le \int_{B} \Big(|\nabla v|^2+\Lambda_1\ind_{\{v>0\}}+\Lambda_2\ind_{\{v<0\}}\Big)\,dx, 
\end{equation}
for every $v \in H^1(B)$ such that $u-v \in H^1_0(B)$.
\end{enumerate}
\end{definition}	

\begin{lemma}[Optimality of the blow-up limits]\label{l:global}
Let $u \in H^1(B_2)$ be as in Theorem \ref{t:constrained} or Theorem \ref{t:two-phase} and let $u_n:=u_{r_n,x_0}$ be a blow-up sequence converging to the blow-up limit $u_0\in H^1_{loc}(\R^2)$. Then, we have:
\begin{enumerate}
\item[{\bf \small (OP)}] If $u$ is as in Theorem \ref{t:constrained} and $x_0\in \partial \Omega_u\cap B_1^+$, then $u_0\circ A_{x_0}^{\sfrac12}$ is a global solution of the one-phase problem with $\Lambda=Q_{\text{\tiny\sc op}}(x_0)$. 
\item[{\bf \small (OP-c)}] If $u$ is as in Theorem \ref{t:constrained} and $x_0\in \partial \Omega_u\cap \partial H\cap B_1$, then, up to a rotation, $u_0\circ A_{x_0}^{\sfrac12}$ is a global solution of the constrained one-phase problem with $\Lambda=Q_{\text{\tiny\sc op}}(x_0)$. 
\item[{\bf \small (TP)}] If $u$ is as in Theorem \ref{t:two-phase} and $x_0\in \partial \Omega_u^+\cap \partial \Omega_u^-\cap B_1$, then $u_0\circ A_{x_0}^{\sfrac12}$ is a global solution of the two-phase problem with $\Lambda_1=Q_{\text{\tiny\sc tp}}^+(x_0)$ and $\Lambda_2=Q_{\text{\tiny\sc tp}}^-(x_0)$.  
\end{enumerate}
\end{lemma}
Recall that the function $\bar{u} = u \circ F_{x_0}$, where $F_{x_0}$ is as in \eqref{e:F_x_0}, is a almost-minimizer of the functional $J_{\text{\tiny\sc op}}^{x_0}$ (Lemma \ref{l:change_of_variables}). We then refer to Lemma 4.6 in \cite{mateve} applied to $\bar{u}$ for the proof of Lemma \ref{l:global}. It is also worth mentioning that the strong convergence of the blow-up sequences and the optimality of the blow-up limits are equivalent. 
\begin{lemma}[Homogeneity of the blow-up limits]\label{prop hom blowup}
Let $u \in H^1(B_2)$ be as in Theorem \ref{t:constrained} or Theorem \ref{t:two-phase}. Let $x_0\in B_1\cap\partial\Omega_u$  and let $u_{x_0,r_n}$ be a blow-up sequence converging to a blow-up limit $u_{0}$. Then, $u_{0}$ is one-homogeneous.
\end{lemma} 

\begin{proof}
Assume that $x_0=0$ and set $\bar u=u\circ F_{x_0}$. Then   
$$u_{x_0,r}=\bar u_r\circ A_{x_0}^{-\sfrac12}\,,\qquad \text{where} \qquad\bar u_r(x):=\frac{\bar u(rx)}{r}.$$ We first notice that by Lemma \ref{l:change_of_variables}, Lemma \ref{l:mono} and the Lipschitz continuity of $u$, we get that the limit $\displaystyle\Theta_\square:=\lim_{r\to0}W_\square(\bar u_r)$, $\square=\text{\sc op}, \text{\sc tp}$, exists and is finite. Now the strong convergence of $\bar u_{r_n}$ to $\bar u_0:=u_0\circ A_{x_0}^{\sfrac12}$ (Proposition \ref{p:convergence}) implies that, for every $s>0$, we have  
$$\Theta_\square:=\lim_{r\to0}W_\square(\bar u_r)=\lim_{n\to\infty}W_\square(\bar u_{r_n})=\lim_{n\to\infty}W_\square(\bar u_{r_ns})=\lim_{n\to\infty}W_\square((\bar u_{r_n})_s)=W_\square((\bar u_0)_s).$$
In particular, $s \mapsto W_\square(\bar{u}_0,s)$ is constant.
Now, since $\bar u_0$ is a global solution (Lemma \ref{l:global}), \eqref{e:der1} and \eqref{e:der2} imply that $E((\bar u_0)_s)=0$, for every $s>0$. Thus we have $x\cdot\nabla\bar{u}_0=\bar{u}_0$ in $\R^2$, which implies that $\bar u_0$ (and thus, $u_0$) is one-homogeneous. 
\end{proof}

\begin{proof}[\bf Proof of Proposition \ref{p:classification}] 
We now notice that $\bar u_0=u_0\circ A_{x_0}^{\sfrac12} :B_1\to\R$ is one-homogeneous and harmonic on the cone $B_1\cap \{\bar u_0\neq0\}$. Thus, the trace of $\bar u_0$ on the sphere satisfies the equation 
$$-\Delta_{\mathbb{S}} \bar u_0=(d-1)\bar u_0\quad\text{on}\quad \mathbb{S}^{d-1}\cap \{\bar u_0\neq0\},$$
where in dimension two the spherical Laplacian $\Delta_{\mathbb{S}}$ is simply the second derivative and $d-1=1$. Thus, $\bar u_0$ is of the form $\bar u_0(\theta)=\sin(\theta+\theta_0)$, $\theta\in \mathbb{S}^2$, for some constant $\theta_0$. This implies that $\{\bar u_0\neq0\}$ is a union of intervals of length $\pi$. In the one-phase case, since $u$ is non-degenerate (see Remark \ref{rk nondeg}), this implies that $\bar u_0$ is of the form \eqref{e:blowup1}, for some constant $\mu(x_0)$. Now, an internal variation argument (see \cite{altcaf}) implies that $\mu(x_0)=Q_{\text{\tiny\sc op}}^{\sfrac12}(x_0)$, if $x_0\in H\cap B_1^+$, and $\mu(x_0)\ge Q_{\text{\tiny\sc op}}^{\sfrac12}(x_0)$, if $x_0\in \partial H\cap B_1$. The two-phase case follows again by an internal variation argument (see \cite{alcafr}).
\end{proof}

Finally, we prove a uniqueness result for the one- and two-phase (Theorem \ref{t:constrained} and Theorem \ref{t:two-phase}) blow-up limits. This is the only result of this section that cannot be immediately extended to higher dimension. This is due to the fact that the epiperimetric inequality (Theorem \ref{t:epi1} and Theorem \ref{t:epi2}) is known (for the moment) only in dimension two.  

\begin{proposition}[Uniqueness of the blow-up and rate of convergence of the blow-up sequences]\label{p:unique}
Let $u:B_2\to \R$ be as in Theorem \ref{t:constrained} or Theorem \ref{t:two-phase}. There are constants $C>0$, $\gamma>0$ and $r_0>0$ such that the following claims do hold.
\begin{enumerate}
\item[{\bf \small (OP)}] If $u$ is as in Theorem \ref{t:constrained}, then for every $x_0\in\partial\Omega_u\cap B_1$, there is a unique blow-up $u_{x_0}:\R^2\to\R$ (of the form \eqref{e:blowup1} or \eqref{e:blowup1c}) such that 
\begin{equation}\label{e:decay1}
\|u_{x_0,r}-u_{x_0}\|_{L^\infty(B_1)}\le Cr^\gamma\quad\text{for every}\quad r\in(0,r_0).
\end{equation}
\item[{\bf \small (TP)}] If $u$ is as in Theorem \ref{t:two-phase}, then for every $x_0\in\partial\Omega_u^+\cap\partial\Omega_u^-\cap B_1$, there is a unique blow-up $u_{x_0}:\R^2\to\R$ (of the form \eqref{e:blowup2}) such that 
\begin{equation}\label{e:decay2}
\|u_{x_0,r}-u_{x_0}\|_{L^\infty(B_1)}\le Cr^\gamma\quad\text{for every}\quad r\in(0,r_0).
\end{equation}
\end{enumerate}
\end{proposition}
\begin{proof}
Let $u$ be as in {\bf \small (OP)} and $x_0\in\partial\Omega_u\cap B_1$. We set $\bar u=u\circ F_{x_0}$ and  $\bar u_r(x):=\frac{\bar u(rx)}{r}$, and we notice that $\bar u_r=u_{x_0,r}\circ A_{x_0}^{\sfrac12}$.  By Lemma \ref{l:change_of_variables} and Lemma \ref{l:mono}, $r\mapsto W_{\text{\tiny\sc op}}(\bar u_{r}) + Cr^\delta$ is monotone. On the other hand, the homogeneity of the blow-up limits, imply that
$$\Theta_{\text{\tiny\sc op}}:=\lim_{r\to0}W_{\text{\tiny\sc op}}(\bar u_{r})=\frac\pi2Q_{\text{\tiny\sc op}}(x_0).$$
Thus, by the epiperimetric inequality (Theorem \ref{t:epi1}), Lemma \ref{l:change_of_variables} and Lemma \ref{l:decay}, we have that there exists a one-homogeneous function $\bar u_0$ such that, for $r>0$ small enough,
$$\|\bar u_{r}-\bar u_{0}\|_{L^2(\partial B_1)}\le Cr^{\gamma_0/2},$$
where $\gamma_0$ is the constant from Lemma \ref{l:decay}. Integrating in $r$, we get that 
$$\|\bar u_{r}-\bar u_{0}\|_{L^2(B_1)}\le Cr^{\gamma_0/2}.$$
Now, since $\bar u_r=u_{x_0,r}\circ A_{x_0}^{\sfrac12}$ and $A_{x_0}^{\sfrac12}$ is invertible, we get 
$$\|u_{x_0,r}-u_{x_0}\|_{L^2(B_1)}\le Cr^{\gamma_0/2},$$
where $u_{x_0}=\bar u_0\circ A_{x_0}^{\sfrac12}$.
Finally, we notice that the Lipschitz continuity of $u$ implies that there is an universal bound on $\|\nabla u_{x_0,r}\|_{L^\infty(B_1)}$ and  $\|\nabla u_{x_0}\|_{L^\infty(B_1)}$. Thus, we get \eqref{e:decay1} with $\gamma = {\gamma_0/4}$. The proof of {\bf \small (TP)} is analogous.
\end{proof}

\begin{remark}\label{r:1-2}
We notice that the above result does not hold at the one-phase points $x_0\in \partial \Omega_u^+\setminus\partial\Omega_u^-$ of the solutions $u$ of the two-phase problem (Theorem \ref{t:two-phase}). This is due to the fact that the positive part $u_+$ is not a solution of the one-phase problem in the balls $B_r(x_0)$ that have non-empty intersection with the negative phase $\Omega_u^-$. In fact, the blow-up limit $u_{x_0}$ (of $u$ at $x_0$) is still unique, but the decay estimate \eqref{e:decay1} holds only for $r<\frac12\text{dist}(x_0,\Omega_u^-)$. 
\end{remark}

\section{Regularity of the one-phase free boundaries. Proof of Theorem \ref{t:constrained}}\label{s:constrained}
Let $u\in H^1(B_2)$, $u\ge 0$, be as in Theorem \ref{t:constrained}. By Proposition \ref{p:unique} we have that, for every $x_0\in \partial\Omega_u\cap B_1$, there is a unique blow-up limit of $u$ at $x_0$. We denote it by 
$$u_{x_0}(x)=\mu(x_0)\max\{0,\nu_{x_0}\cdot x\},$$
where $\nu_{x_0}$ is of the form $A_{x_0}^{\sfrac12}[\nu]$, for some $\nu\in \partial B_1$; and $\mu(x_0)$ is such that $Q_{\text{\tiny\sc op}}(x_0)\le \mu^2(x_0)\le M_{\text{\tiny\sc A}}L^2$, where $L$ is the Lipschitz constant of $u$. We also notice that 
$$\mu(x_0)=Q_{\text{\tiny\sc op}}^{\sfrac12}(x_0)\quad\text{whenever}\quad x_0\in \partial\Omega_u\cap B_1^+.$$
Moreover, for every point $x_0\in\partial\Omega_u\cap B_1$, we define the half-plane
$$H_{x_0}:=\{x\in\R^2\,:\,x\cdot\nu_{x_0}>0\}.$$
We first prove the following: 
 
\begin{lemma}\label{l:flatness1}
Let $u$ be as in Theorem \ref{t:constrained}. There are constants $C>0$, $\gamma>0$ and $r_0>0$ such that, for every $x_0\in\partial\Omega_u\cap B_1$, we have 
\begin{equation}\label{e:flatness1}
\Omega_{{x_0,r}}\cap B_1\supset \{x\in B_1\,:\, x\cdot\nu_{x_0}>Cr^\gamma\}\quad\text{and}\quad \Omega_{{x_0,r}}\cap \{x\in B_1\,:\, x\cdot\nu_{x_0}<-Cr^\gamma\}=\emptyset,
\end{equation}
for every $r\in(0,r_0)$, where $\Omega_{x_0,r}:=\{u_{x_0,r}>0\}$. 
\end{lemma}
\begin{proof}
The first part of \eqref{e:flatness1} follows by the uniform convergence of the blow-up sequence $u_{x_0,r}$ (Proposition \ref{p:unique}, equation \eqref{e:decay1}) and the form of the blow-up limit $u_{x_0}$. The second part of \eqref{e:flatness1} follows again by \eqref{e:decay1}, the fact that $u_{x_0}\equiv 0$ on $B_1\setminus H_{x_0}$ and by the non-degeneracy of $u$, which can be written as
$$\text{If}\quad u_{x_0,r}(y_0)>0\,,\quad\text{then}\quad \|u_{x_0,r}\|_{L^\infty(B_s(y_0))}\ge Cs\,,\quad\text{for every}\quad s\in(0,1),$$
for some $C>0$. 
\end{proof}
\begin{lemma}\label{l:oscillation1}
Let $u$ be as in Theorem \ref{t:constrained}. There are constants $R,\alpha\in(0,1)$ and $C>0$ such that, for every $x_0,y_0\in\partial\Omega_u\cap B_R$, we have 
\begin{equation}\label{e:oscillation1}
|\nu_{x_0}-\nu_{y_0}|\le C|x_0-y_0|^\alpha\qquad \text{and}\qquad |\mu(x_0)-\mu(y_0)|\le C|x_0-y_0|^\alpha.
\end{equation}
\end{lemma}
\begin{proof}
Let $\gamma\in(0,1)$ be the exponent from Proposition \ref{p:unique} and let $\alpha:=\frac{\gamma}{1+\gamma}$. 
Let $x_0,y_0 \in B_{R} \cap \partial \Omega_u$, where we choose $R$ such that $(2R)^{1-\alpha}\le r_0$, where $r_0$ is the constant from Proposition \ref{p:unique}. We set $r:=|x_0-y_0|^{1-\alpha}$. Recall that $u$ is Lipschitz continuous and set $L=\|\nabla u\|_{L^\infty}$. Then, for every $x \in B_1$, we have 
\[|u_{x_0,r}(x) - u_{y_0,r}(x)| = \frac1r|u(x_0+rx)-u(y_0+rx)|\le  L\frac{|x_0-y_0|}r=L |x_0-y_0|^{\alpha}.\]
and then, by an integration on $B_1$, we get
$$\|u_{x_0,r} - u_{y_0,r}\|_{L^2(B_1)} \leq |B_1|^{\sfrac12} L |x_0-y_0|^{\alpha}.$$
On the other hand, by the choice of $R$, we have that $r\le r_0$; applying Proposition \ref{p:unique}, we get  
$$\|u_{x_0,r} - u_{x_0}\|_{L^2(B_1)}\le Cr^\gamma\qquad\text{and}\qquad \|u_{y_0,r} - u_{y_0}\|_{L^2(B_1)} \le Cr^\gamma.$$
Thus, by the triangular inequality and the fact that $r^\gamma=|x_0-y_0|^\alpha$, we obtain 
\begin{equation}\label{e:osc_blowup1}
\|u_{x_0} - u_{y_0}\|_{L^2(B_1)} \leq \big(|B_1|^{\sfrac12}L +2C\big)|x_0-y_0|^{\alpha}.
\end{equation}
The conclusion now follows by a general argument. Indeed, for any pair of vectors $v_1,v_2 \in \R^2$, we have
\begin{align}\label{eq 6}
\nonumber |v_1 - v_2| &= \left( \frac{2}{\pi} \int_{B_1} |v_1 \cdot x - v_2 \cdot x|^2 \,dx \right)^{\sfrac12} \\
&\leq \left( \int_{B_1} |(v_1 \cdot x)_+ - (v_2 \cdot x)_+|^2 \,dx \right)^{\sfrac12} + 
\left(\int_{B_1} |(v_1 \cdot x)_- - (v_2 \cdot x)_-|^2 \,dx \right)^{\sfrac12} \\
\nonumber& = 2 \left(\int_{B_1} |(v_1 \cdot x)_+ - (v_2 \cdot x)_+|^2 \,dx \right)^{\sfrac12}.
\end{align}
Applying the above estimate to $v_1=\mu(x_0)\nu_{x_0}$ and $v_2=\mu(y_0)\nu_{y_0}$, and using \eqref{e:osc_blowup1}, we get \eqref{e:oscillation1}. \end{proof}

\begin{proof}[\bf Proof of Theorem \ref{t:constrained}]
We first claim that, for every $\eps>0$, there exists $\rho>0$ such that, for $x_0 \in \partial \Omega_u \cap B_{\rho}$ we have 
\begin{equation}\label{eq cone}
u>0\,\text{ on }\, C^+(x_0,\varepsilon) \cap B_\rho(x_0)\, \qquad\text{and}\qquad u=0 \,\text{ on }\, C^-(x_0,\varepsilon) \cap B_\rho(x_0),
\end{equation}
where 
\begin{equation*}
C^\pm(x_0,\varepsilon) := \left\{ x \in \R^2 \b \{0\} : \pm \nu_{x_0} \cdot (x-x_0) \geq \varepsilon |x-x_0| \right\}.
\end{equation*}
Indeed, the flatness estimate \eqref{e:flatness1} implies \eqref{eq cone} by taking $\rho$ such that $C\rho^\gamma\le \eps$, where $C$ and $\gamma$ are the constants from Lemma \ref{l:flatness1}.

We now fix $x_0\in B_1\cap\partial\Omega_u$. Without loss of generality we can suppose that $x_0=0$ and $H_{x_0}=\{(s,t)\in\R^2\,:\,t>0\}$. Now, let $\eps \in (0,1)$ and $\rho>0$ as in \eqref{eq cone} and set $\delta = \rho\sqrt{1-\eps^2}$. By \eqref{eq cone} we have for every $s\in(-\delta,\delta)$

 $\,\bullet$ the set $\mathcal S_+^s:=\{t\in(-\delta,\delta)\ :\ u(s,t)>0\}$ contains the interval $(\rho\eps,\delta)$; 

 $\,\bullet$ the set $\mathcal S_0^s:=\{t\in(-\delta,\delta)\ :\ u(s,t)=0\}$ contains the interval $(-\delta,-\rho\eps)$.

\noindent This implies that the function 
\[ g(s) := \max \{ t \in \R : u(s,t)>0\} \]
is well defined and such that 
$$SQ_\delta\cap \Omega_u=\{(s,t)\in SQ_\delta\,:\, g(s)<t\}\qquad \text{and}\qquad SQ_\delta\setminus \Omega_u=\{(s,t)\in SQ_\delta\,:\, g(s)\geq t\},$$
where $SQ_\delta=(-\delta,\delta)\times(-\delta,\delta)$. Now, the flatness condition \eqref{e:flatness1} implies that $g$ is differentiable on $(-\delta,\delta)$.
 Furthermore, since $\nu$ is H\"{o}lder continuous, we deduce that $g$ is a function of class $C^{1,\alpha}$. This concludes the proof. 
\end{proof}

\section{Regularity of the two-phase free boundaries. Proof of Theorem \ref{t:two-phase}}\label{s:two-phase}
Let $u$ be as in Theorem \ref{t:two-phase}. Then, by Proposition \ref{p:unique}, at every point $x_0\in\partial\Omega_u\cap B_1$ there is a unique blow-up limit $u_{x_0}$ given by
\begin{align*}
u_{x_0}(x) = \mu_+(x_0) \max \{0,x\cdot\nu_{x_0}\}, \quad &\text{if } x_0 \in \Gamma_+ := (\partial \Omega_u^+ \backslash \partial \Omega_u^-) \cap B_1; \\
u_{x_0}(x) = \mu_-(x_0) \min \{0,x\cdot\nu_{x_0}\}, \quad &\text{if } x_0 \in \Gamma_- := (\partial \Omega_u^- \backslash \partial \Omega_u^+) \cap B_1; \\
u_{x_0}(x) = \mu_+(x_0) \max \{ 0,x\cdot\nu_{x_0}\} + \mu_-(x_0) \min \{ 0,x\cdot\nu_{x_0}\}, \quad &\text{if } x_0 \in \Gamma_{\text{\tiny\sc tp}} := \partial \Omega_u^+ \cap \partial \Omega_u^-\cap B_1,
\end{align*}
where $\nu_{x_0} \in \R^2$ is of the form $A_{x_0}^{-\sfrac12}[\nu]$, for some $\nu\in\partial B_1$, and $\mu_+(x_0)$ and $\mu_-(x_0)$ are positive and such that $Q_{\text{\tiny\sc tp}}^\pm(x_0)\le \mu_\pm^2(x_0) \le M_{\text{\tiny\sc A}}L^2$, where $L=\|\nabla u\|_{L^\infty(B_2)}$ is the Lipschitz constant of $u$, and 
\begin{align*}
\mu_\pm^2(x_0)=Q_{\text{\tiny\sc tp}}^\pm(x_0), \quad &\text{if } x_0 \in \Gamma_\pm \\
\mu_+^2(x_0)-\mu_-^2(x_0) = Q_{\text{\tiny\sc tp}}^+(x_0)-Q_{\text{\tiny\sc tp}}^-(x_0), \quad &\text{if } x_0 \in \Gamma_{\text{\tiny\sc tp}}.
\end{align*}

Notice that Corollary \ref{t:unconstrained} already implies that the one-phase free boundaries $\Gamma_+$ and $\Gamma_-$ are $C^{1,\alpha}$ regular. Thus, it remains to prove that $\partial\Omega_u^+$ and $\partial\Omega_u^-$ are smooth in a neighborhood of $\Gamma_{\text{\tiny\sc tp}}$.

\begin{lemma}[Flatness of the free boundary at the two-phase points]\label{l:flatness2}
Let $u$ be as in Theorem \ref{t:two-phase}. There are constants $C>0$, $\gamma>0$ and $r_0>0$ such that, for every $x_0\in\partial\Gamma_{\text{\tiny\sc tp}}$, we have 
\begin{equation}\label{e:flatness2}
\Omega_{{x_0,r}}^+\cap B_1\supset \{x\in B_1\,:\, x\cdot\nu_{x_0}>Cr^\gamma\}\quad\text{and}\quad \Omega_{{x_0,r}}^-\cap B_1\supset  \{x\in B_1\,:\, x\cdot\nu_{x_0}<-Cr^\gamma\},
\end{equation}
for every $r\in(0,r_0)$, where $\Omega_{x_0,r}^+:=\{u_{x_0,r}>0\}$ and $\Omega_{x_0,r}^-:=\{u_{x_0,r}<0\}$. 
\end{lemma}
\begin{proof}
Both the inclusions of \eqref{e:flatness2} follow by the uniform convergence of $u_{x_0,r}$ (Proposition \ref{p:unique}, equation \eqref{e:decay2}) to the blow-up limit $u_{x_0}$. 
\end{proof}
\begin{lemma}\label{l:oscillation2}
Let $u$ be as in Theorem \ref{t:two-phase}. There are constants $R,\alpha\in(0,1)$ and $C>0$ such that, for every $x_0,y_0\in\partial\Gamma_{\text{\tiny\sc tp}}\cap B_R$, we have  
\begin{equation}\label{e:oscillation2}
|\nu_{x_0}-\nu_{y_0}|\le C|x_0-y_0|^\alpha\qquad \text{and}\qquad |\mu_{\pm}(x_0)-\mu_{\pm}(y_0)|\le C|x_0-y_0|^\alpha.
\end{equation}
\end{lemma}
\begin{proof}
The proof follows step by step the one of Lemma \ref{l:oscillation1}. 
\end{proof}

Reasoning as in the one-phase case, and using Lemma \ref{l:flatness2} and Lemma \ref{l:oscillation2}, one can prove that the two-phase free boundary $\Gamma_{\text{\tiny\sc tp}}$ is {\it contained} in a $C^{1,\alpha}$ curve. Unfortunately, this result by itself is not sufficient to deduce that $\partial\Omega_u^\pm$ are smooth. We now prove that the function $u_+$ (resp. $u_-$) is a solution of the one-phase free boundary problem
\begin{equation}\label{e:viscosity}
-\text{div}(A\nabla u_+)=f_+\quad\text{in}\quad \Omega_u^+,\qquad |A_{x_0}^{\sfrac12}\nabla u_+|(x_0)=\mu_+(x_0)\quad\text{for every}\quad x_0\in\partial\Omega_u^+
\end{equation}
where the boundary equation is understood in a classical sense. This is an immediate consequence of the following lemma which states that $u_+$ is differentiable in  $\Omega_u^+$ up to the boundary. 

\begin{lemma}[Differentiability at points of the free boundary]\label{l:differentiability}
Let $u$ be as in Theorem \ref{t:two-phase}.\\ We consider two cases. 

\noindent{\bf \small(OP)}. For every $x_0 \in(\partial\Omega_u^+\setminus \partial\Omega_u^-)\cap B_1$, $u_+$ is differentiable at $x_0$ and there is $r(x_0)>0$ such that  
$$ \big|u_+(x) - \mu_+(x_0)(x-x_0)\cdot \nu_{x_0}\big| \leq C|x-x_0|^{1+\gamma}\quad\text{for every} \quad x \in B_{r(x_0)}(x_0) \cap \Omega_u^+.$$
\noindent{\bf \small(TP)}. There exists a universal constant $r_0>0$ such that for every $x_0 \in\partial\Omega_u^+\cap \partial\Omega_u^-\cap B_1$, $u_+$ is differentiable in $B_{r_0}(x_0) \cap \Omega_u^+$, and
\begin{equation}\label{e:differentiabilityTP}
\big|u_+(x) - \mu_+(x_0)(x-x_0)\cdot \nu_{x_0}\big| \leq C|x-x_0|^{1+\gamma}\quad\text{for every} \quad x \in B_{r_0}(x_0) \cap \Omega_u^+.
\end{equation}
In particular, for every $x_0 \in\partial\Omega_u^+\cap B_1$, we have $\nabla u_+(x_0)=\mu_+(x_0)\nu_{x_0}$.
\end{lemma}
\begin{proof}
The two cases are analogous. We will prove {\bf \small(TP)}. By Proposition \ref{p:unique}, for every $r<r_0$, we have 
\begin{equation*}
\|\max\{0,u_{x_0,r}\}-\max\{0,u_{x_0}\}\|_{L^\infty(B_1)}\le \|u_{x_0,r}-u_{x_0}\|_{L^\infty(B_1)} \leq Cr^{\gamma}.
\end{equation*}
Thus, using the flatness of the free boundary (Lemma \ref{l:flatness2}), we get for every $x\in B_1\cap\{u_{x_0,r}>0\}$
\begin{multline*}\label{eq Linfty esti}
|\max\{0,u_{x_0,r}(x)\}-\mu_+(x_0)x\cdot \nu_{x_0}|\le |\max\{0,u_{x_0,r}(x)\}-\max\{0,u_{x_0}(x)\}| \\
+ \mu_+(x_0)|\min\{0,x\cdot\nu_{x_0}\}| \leq Cr^{\gamma}.
\end{multline*}
Now, taking $r=|x-x_0|$ and rescaling the above inequality, we obtain \eqref{e:differentiabilityTP} 
\end{proof}
We notice that at the two-phase free boundary point the estimate \eqref{e:differentiabilityTP} holds in a ball whose  radius does not depend on the point. Moreover, on the two-phase free boundary the gradient has a universal modulus of continuity (see Lemma \ref{l:oscillation2}). 
We next show that $\mu_+$ is H\"older continuous on $\partial\Omega_u^+$. 

\begin{lemma}\label{l:holder}
The function $\mu_+: \partial \Omega_u^+ \rightarrow \R$ is (locally) H\"older continuous.
\end{lemma}

\begin{proof}
We infer that lemma \ref{l:holder} is a consequence of the following claim: if $(x_n)_{n\ge1}$ is a sequence of one-phase points, $x_n\in \Gamma_+$, converging to a two-phase point $y_0 \in \Gamma_{\text{\tiny\sc tp}}$, then $\mu_+(y_0)=Q_+^{\sfrac12}(y_0)$, where we set $Q_+:=Q_{\text{\tiny\sc tp}}^+$.
Indeed, we we notice that: 

$\bullet$ on the set $\Gamma_+$, we have $\mu_+=Q_+^{\sfrac12}$. 

$\bullet$ for every $y_1,y_2\in \Gamma_{\text{\tiny\sc tp}}$, we have $|\mu_+(y_1)-\mu_+(y_2)|\le C|y_1-y_2|^\alpha$.
\smallskip

\noindent By the first bullet $\mu_+$ is H\"older continuous on the open subset $\Gamma_+$ of $\partial\O_u\cap B_1$. To prove that $\mu_+$ is H\"older continuous on $\Gamma_{\text{\tiny\sc tp}}$, let $y_0\in \Gamma_{\text{\tiny\sc tp}}$ and $x,y\in\partial\O^+_u\cap B_{r_0}(y_0)$; we have to show that $|\mu_+(x)-\mu_+(y)|\leq C|x-y|^\alpha$. By the two bullets, it is obvious if either $x,y \in \Gamma_+$ or $x,y\in \Gamma_{\text{\tiny\sc tp}}$. Then, assume that $x\in \Gamma_+$ and $y\in \Gamma_{\text{\tiny\sc tp}}$ and denotes by $y_1$ the projection of $y$ on the closed set $\Gamma$, where $\Gamma$ is the set of points $z$ in $\Gamma_{\text{\tiny\sc tp}}$ such that every neighborhood of $z$ has non-empty intersection with $\Gamma_+$. Note that, by definition of $y_1$, we have $|y_1-y|\leq|x-y|$ and then $|x-y_1|\leq2|x-y|$. Therefore, using the triangular inequality and the claim we get
\[|\mu_+(x)-\mu_+(y)|\leq |Q_+^{\sfrac12}(x)-Q_+^{\sfrac12}(y_1)|+|\mu_+(y_1)-\mu_+(y)| \leq C|x-y|^\alpha.\]

We now prove the claim. Up to a linear change of coordinates we may suppose that $A_{y_0}=Id$.   
Denote by $y_n$ the projection of $x_n$ on the closed set $ \partial \Omega_u^+\cap \partial \Omega_u^-$ and set $r_n:=|x_n-y_n|$. Since $u$ is Lipschitz continuous, up to a subsequence, $u_n:=u_{x_n,r_n}^+$ converges locally uniformly to some function $u_\infty$. The absence of two-phase points in $B_{r_n}(x_n)$ implies that $u_n$ is a solution of
\[ -\text{div}(A_n\nabla u_n) = r_n f_n\quad \text{in}\quad \{u_n>0\}\cap B_1\,, \qquad |\nabla u_n| = q_n\quad\text{on}\quad \partial \{u_n>0\}\cap B_1\,, \]
where $A_n(x):=A(x_n+r_nx)$, $f_n(x):=f_+(x_n+r_nx)$ and $q_n(x)=Q_+^{\sfrac12}(x_n+r_nx)|\nu_{x_n+r_nx}|$, where we recall that $\nu_{x_n+r_nx}$ is of the form $A_{x_n+r_nx}^{\sfrac12}[\tilde\nu]$, for some $\tilde \nu\in\partial B_1$. Passing to the limit as $n\to\infty$, we obtain that $u_\infty$ is a viscosity solution to
\[ -\Delta u_\infty = 0\quad\text{in}\quad \{u_\infty>0\}\cap B_1\,, \qquad |\nabla u_\infty| = Q_+^{\sfrac12}(y_0)|\nu_{y_0}|\quad \text{on}\quad \partial \{u_\infty>0\}\cap B_1\,. \]
On the other hand, for every $\xi \in B_1$, we have
\[ u_{x_n,r_n}(\xi) = u_{y_n,r_n}(\xi + \xi_n)\,,\quad \text{where}\quad \xi_n := \frac{x_n-y_n}{r_n}\in \partial B_1\,, \]
and, up to a subsequence, we can assume that $\xi_n$ converges to some $\xi_\infty \in \partial B_1$.
Since $y_n \in \partial \Omega_u^+\cap \partial \Omega_u^-$, Lemma \ref{l:differentiability} implies that, for every $x \in B_{2r_n}(y_n) \cap \{u>0\}$, we have 
\[ |u(x)-\mu_+(y_n)\max\{0,(x-y_n)\cdot \nu_{y_n}\}| \leq C |x-y_n|^{1+\gamma} \leq Cr_n^{1+\gamma}. \]
After rescaling, this gives
$$ \big|u_{y_n,r_n}(\xi+\xi_n) - \mu_+(y_n)\max\{0,(\xi+\xi_n)\cdot \nu_{y_n}\}\big| \leq Cr_n^{\gamma}\quad \text{for every}\quad \xi \in B_1 \cap \{u_{x_n,r_n}>0\}. $$
Moreover, by the continuity of $\mu_+$ on $\partial \Omega_u^+\cap \partial \Omega_u^-$, we have that, for every $\xi\in B_1$,
$$\lim_{n\to\infty}\big|\mu_+(y_n)\max\{0,(\xi+\xi_n)\cdot \nu_{y_n}\} - \mu_+(y_0)\max\{0,(\xi+\xi_\infty)\cdot \nu_{y_0}\}\big| =0.$$
Therefore, it follows that $u_{x_n,r_n}(\xi)=u_{y_n,r_n}(\xi+\xi_n)$ converges to 
$$u_\infty(\xi)=\mu_+(y_0)\max\big\{0,(\xi+\xi_\infty)\cdot \nu_{y_0}\big\}\quad\text{for every}\quad \xi \in B_1.$$
Next we claim that $\xi_\infty \cdot \nu_{y_0} =0$. Indeed, if $\xi_\infty \cdot \nu_{x_0}>0$, then $u_\infty(0)>0$ which is in contradiction with the uniform convergence of $u_n$; on the other hand, if $\xi_\infty \cdot e_{x_0}<0$, then  $u_\infty\equiv0$ in a neighborhood of zero, which is in contradiction with the non-degeneracy of $u_n$. Thus, we get 
$$u_\infty(\xi)=\mu_+(y_0)\max\big\{0, \xi\cdot \nu_{y_0}\big\}\quad\text{for every}\quad \xi \in B_1.$$
Now since $|\nabla u_\infty|=\mu_+(y_0)$, we get that $\mu_+(y_0)=Q_+^{\sfrac12}(y_0)$.
\end{proof}

Theorem \ref{t:two-phase} is now a consequence of \eqref{e:viscosity}, the Lemma \ref{l:holder} and a general result (Theorem \ref{t:desilva}) on the regularity of the one-phase flat free boundaries, which is due to De Silva (see \cite{desilva}). In the appendix we state Theorem \ref{t:desilva} in its full generality, for viscosity solutions of the problem \eqref{e:viscosity}, but in our case the function $u_+$ is a classical solution, differentiable everywhere on $\overline\Omega_u^+$. 

\section{Proof of Theorem \ref{t:multi}}\label{s:multi}
\subsection{Preliminary results} In this subsection, we briefly recall the known results on the problem \eqref{e:multi}. 
The existence of a solution of \eqref{e:multi} in the class of the almost-open subsets of $\mathcal D$ can be proved by a general variational argument (we refer to \cite{bucve} and to the book \cite{bubu05} for more details). In the context of open sets, the existence of an optimal $n$-uple was proved in \cite{benve}. 

From now on, $(\Omega_1,\dots,\Omega_n)$ will be a solution of \eqref{e:multi} and $u_i:\R^d\to\R$, for $i=1,\dots,n$, will denote the first normalized eigenfunction of the Dirichlet Laplacian on $\Omega_i$, that is, 
$$-\Delta u_i=\lambda_1(\Omega_i) u_i\quad\text{in}\quad\Omega_i\,,\qquad u_i=0\quad\text{on}\quad \R^2\setminus\Omega_i\,,\qquad \int_{\Omega_i} u_i^2\,dx=1,$$
where, for every $i=1,\dots,n$, 
$$\lambda_1(\Omega_i)=\min_{u\in H^1_0(\Omega_i)\setminus\{0\}}\frac{\int_{\Omega_i} |\nabla u|^2\,dx}{\int_{\Omega_i} u^2\,dx}=\frac{\int_{\Omega_i} |\nabla u_i|^2\,dx}{\int_{\Omega_i} u_i^2\,dx},$$
where $H^1_0(\Omega_i)=\{u\in H^1(\R^d)\,:\,u=0\ \text{on}\ \R^2\setminus\Omega_i\}$. In particular,  
$u_i\ge 0$ on $\R^2$ and $\Omega_i=\{u_i>0\}$. 
\medskip

\noindent{\it Lipschitz continuity.} The functions $u_i:\R^2\to\R$ are Lipschitz continuous on $\R^2$, that is, there is a universal constant $L>0$ such that $\|\nabla u_i\|_{L^\infty(\R^2)}\le L$, for every $i=1,\dots,n$. We refer to \cite{bucve} for the general case and to \cite{benve} for a simplified version in dimension two. 
\medskip


\noindent{\it Absence of triple points.} For every $1\le i<j<k\le n$, we have that $\partial\Omega_i\cap\partial\Omega_j\cap\partial\Omega_k=\emptyset$ (see \cite{bucve} and \cite{benve} for a simpler proof in dimension two). 
\medskip

\noindent{\it Absence of two-phase points on the boundary of the box.} For every $1\le i<j\le n$, we have that $\partial\Omega_i\cap\partial\Omega_j\cap\partial \mathcal D=\emptyset$ (see \cite{benve}). 
\medskip

As a consequence of the above properties, we have that, for every $i\in\{1,\dots,n\}$, the boundary $\partial\Omega_i$ can be decomposed as follows: 
$$\partial\Omega_i= \bigcup_{k\neq i}(\partial\Omega_i\cap\partial\Omega_k)\cup (\partial\Omega_i\cap\partial \mathcal D)\cup\Gamma_{\text{\tiny\sc op}}(\Omega_i),$$
where $\Gamma_{\text{\tiny\sc op}}(\Omega_i)$ is the one-phase free boundary of $\Omega_i$, determined by: 
$$x_0\in \Gamma_{\text{\tiny\sc op}}(\Omega_i)\Leftrightarrow \text{there exists}\ r>0\ \text{such that}\ B_r(x_0)\cap \Big((\R^2\setminus \mathcal D)\cup \bigcup_{k\neq i}\Omega_k\Big)=\emptyset.$$ 
We notice that already using the the regularity result of Brian\c con and Lamboley \cite{brla}, the one-phase free boundary (lying inside the open set $\mathcal D$) is locally a $C^{1,\alpha}$ curve. Thus, in order to prove Theorem \ref{t:multi}, it will be sufficient to show that $\partial \Omega_i$ is $C^{1,\alpha}$ in a neighborhood of the points of $\partial\Omega_i\cap\partial \mathcal D$ (Subsection \ref{sub:constrained}) and $\partial\Omega_i\cap\partial\Omega_k$ (Subsection \ref{sub:two-phase}). 

\subsection{One-phase points at the boundary of the box}\label{sub:constrained}
Let $1\le i\le n$ and $x_0\in\partial D\cap\partial\Omega_i$. Then, there is a neighborhood $\mathcal U$ of $x_0$ such that $\mathcal U\cap\Omega_j=\emptyset $, for every $j\neq i$. For the sake of simplicity, in this subsection, we will set 
$$\Omega=\Omega_i\,,\quad u=u_i\,,\quad x_0=0\quad\text{and}\quad D=\Omega_i\cup(\mathcal D\cap\mathcal U).$$ 
It is well known that the eigenvalues of the Dirichlet Laplacian are variationally characterized by
\[ \lambda_1(\O) = \int_\O |\nabla u|^2\,dx = \min \left\{\int_\O |\nabla v|^2\,dx : v \in H^1_0(\O), \int_\O v^2\,dx =1\right\}.\]
Moreover, $\{u>0\}=\O$ and $u$ is a solution of the following minimization problem:
\begin{equation}\label{e:eq_min_u}
\min \left\{ \int_D |\nabla v|^2\,dx + \Lambda |\{v>0\}| : v \in H^1_0(D), \ \int_D v^2\,dx =1 \right\}.
\end{equation}
We will show that the solution $u$ of \eqref{e:eq_min_u} is an almost-minimizer of the one-phase functional $J_{\text{\tiny\sc op}}$. 
A result in the same spirit was proved given in a more general case in \cite[Proposition 2.1]{mateve}.

\begin{lemma}[Almost-minimality of the eigenfunction]\label{l:almost_almost_u}
Let $u:\R^d\to\R$ be a Lipschitz continuous function, $L=\|\nabla u\|_{L^\infty}$ be the Lipschitz constant of $u$ and $\displaystyle\lambda_1(\Omega_u)=\int_D|\nabla u|^2\,dx$. If $u$ is a solution of the minimization problem \eqref{e:eq_min_u}, then there exists $r_0>0$ such that $u$ satisfies the following almost-minimality condition:
\medskip

\noindent For every $r\in(0,r_0)$ and $x_0\in\partial\Omega_u$,
\begin{align*}\label{e:almost_almost_u}
\int_{B_r(x_0)} |\nabla u|^2\,dx& + \Lambda |\Omega_u\cap B_r(x_0)| \le \big(1+C_1r^{d+2}\big)\int_{B_r(x_0)} |\nabla v|^2\,dx +\Lambda|\Omega_v\cap B_r(x_0)|+C_2r^{d+2},
\end{align*}
for every $v \in H^1_0(D)$ such that $u=v$ on $D\setminus B_r(x_0)$, where $C_1=2L^2$ and $C_2=\lambda_1(\Omega_u)2L^2$.
\end{lemma}

\begin{proof}
Let $x_0\in\partial\Omega_u$, $r>0$ and $v\in H^1_0(D)$ be such that $u=v$ on $D\setminus B_r(x_0)$. Then, define the renormalization $w = \|v\|_{L^2}^{-1}v \in H^1_0(D)$ and notice that we have  
\begin{align*}
\int_D|\nabla w|^2\,dx&=\Big(\int_Dv^2\,dx\Big)^{-1}{\int_D|\nabla v|^2\,dx}\le \Big(1-\int_{B_r(x_0)}u^2\,dx\Big)^{-1}{\int_D|\nabla v|^2\,dx}\\
&\le \frac1{1-L^2r^{d+2}}{\int_D|\nabla v|^2\,dx}\le\big(1+2L^2r^{d+2}\big){\int_D|\nabla v|^2\,dx},
\end{align*}
where for the last inequality, we choose $r_0$ such that $2L^2r_0^{d+2}\le 1$ and we use the inequality $\displaystyle\frac{1}{1-X}\le 1+2X$, for every $X\le \sfrac12$, with $X=L^2r^{d+2}$. Now use $w$ as a test function in \eqref{e:eq_min_u} to get that 
\begin{equation}\label{e:almost_almost_uhuhu}
\int_D |\nabla u|^2\,dx + \Lambda |\{u>0\}| \leq \big(1+2L^2r^{d+2}\big)\int_D |\nabla v|^2\,dx +\Lambda|\{v>0\}|,
\end{equation}
from which the claim easily follows since $\displaystyle \int_{D\setminus B_r(x_0)} |\nabla v|^2\,dx=\int_{D\setminus B_r(x_0)} |\nabla u|^2\,dx\le \lambda_1(\Omega_u)$. 
\end{proof}

We now notice that the $C^2$ regularity of $\partial\mathcal D$ implies that there is a constant $\delta>0$ and a function $g:(-\delta,\delta)\to\R$ such that 
$$\mathcal D\cap SQ_{\delta}=\big\{(x_1,x_2)\in SQ_{\delta}\,:\, g(x_1)<x_2\big\},$$
where $SQ_{\delta}=(-\delta,\delta)\times (-\delta,\delta)$. Moreover, up to a rotation of the plane, we can assume that $g'(0)=0$. 
Let $\psi : SQ_{\delta}\subset \R^2 \rightarrow \R^2$ be the function that straightens out the boundary of $\mathcal D$ and let $\phi = \psi^{-1} : \psi( SQ_{\delta}) \subset \R^2 \rightarrow \R^2$ be its inverse: 
\[ \psi(x_1,x_2)=(x_1,x_2-g(x_1)), \quad \phi(x_1,x_2)=(x_1,x_2+g(x_1)). \]
We define the matrix-valued function $A=(a_{ij})_{ij}:SQ_\delta\to M^2(\R)$ by
\[
A_x:=\begin{pmatrix}
a_{11}(x) & a_{12}(x) \\
a_{21}(x) & a_{22}(x)
\end{pmatrix}= 
\begin{pmatrix}
1 & -g'(x_1) \\
-g'(x_1) & 1 + (g'(x_1))^2
\end{pmatrix}
\quad\text{for every}\quad x=(x_1,x_2)\in SQ_\delta. \]
We recall that $H=\{(x_1,x_2)\in\R^2 : x_2>0\}$. By an elementary change of coordinates, we obtain the following result. 
\begin{lemma}\label{l:almost_almost_tilde_u}
Let $u$ and $A$ be as above. There exist constants $C_1,C_2>0$ and $r_0>0$ such that $B_{2r_0}\subset \psi(SQ_\delta)$ and the function $\tilde{u} := u \circ \phi$ satisfies the following almost-minimality condition:
\begin{itemize}[label=]
	\item For every $x_0\in \partial\O_{\tilde{u}}\cap B_{r_0}$ and $r\in(0,r_0)$ we have
	\begin{multline*}
	\int_{B_r(x_0)} a_{ij}(x)\,\frac{\partial\tilde{u}}{\partial x_i}\,\frac{\partial \tilde{u}}{\partial x_j} \,dx + \Lambda |\Omega_{\tilde u} \cap B_r(x_0)| \\
	\leq (1+C_1r^{d+2})\int_{B_r(x_0)} a_{ij}(x)\,\frac{\partial\tilde{v}}{\partial x_i}\,\frac{\partial \tilde{v}}{\partial x_j}\,dx + \Lambda |\Omega_{\tilde v} \cap B_r(x_0)| + C_2r^{d+2},
	\end{multline*} 
	for every $\tilde{v} \in H^1(B_{2r_0})$ such that $\tilde{u}=\tilde{v}$ on $B_{2r_0}\setminus B_r(x_0)$ and $\Omega_{\tilde v}\subset H$.
\end{itemize}
\end{lemma}

\begin{proof}
Let $x_0\in B_{r_0}$, $r\in(0,r_0)$ and $\tilde{v}$ such that $\tilde{u}=\tilde{v}$ on $B_{2r_0}\setminus B_r(x_0)$. Then, use $v \in H^1_0(D)$ defined by $v=\tilde{v}\circ\psi$ in $\psi^{-1}(B_{2r_0})$ and $v=u$ otherwise, as a test function in Lemma \ref{l:almost_almost_u} to get
\[ \int_{B_{c_\phi r}(y_0)} |\nabla u|^2\,dx + \Lambda|\Omega_u\cap B_{c_\phi r}(y_0)| \leq (1+C_1r^{d+2})\int_{B_{c_\phi r}(y_0)}|\nabla v|^2\,dx + \Lambda|\Omega_v\cap B_{c_\phi r}(y_0)| + Cr^{d+2}, \]
where $c_\phi$ is a positive constant depending only on $\phi$ such that $\phi(B_r(x_0))\subset B_{c_\phi r}(y_0)$ and $y_0=\phi(x_0)$. Now, with a change of coordinates and noticing that $u=v$ on $\phi(B_r(x_0))$ we have
\begin{align*}
\int_{B_r(x_0)} a_{ij}(x)\,\frac{\partial\tilde{u}}{\partial x_i}\,\frac{\partial \tilde{u}}{\partial x_j} \,dx + &\Lambda |\Omega_{\tilde u} \cap B_r(x_0)| 
= \int_{\phi(B_r(x_0))}|\nabla u|^2\,dx + \Lambda|\Omega_u\cap\phi(B_r(x_0))| \\
&\leq (1+C_1r^{d+2})\int_{\phi(B_r(x_0))}|\nabla v|^2\,dx + \Lambda|\Omega_v\cap \phi(B_r(x_0))| + C_2r^{d+2} \\
&= (1+C_1r^{d+2})\int_{B_r(x_0)} a_{ij}(x)\,\frac{\partial\tilde{v}}{\partial x_i}\,\frac{\partial \tilde{v}}{\partial x_j}\,dx + \Lambda |\Omega_{\tilde v} \cap B_r(x_0)| + C_2r^{d+2},
\end{align*}
where $C_2 = \lambda_1(\Omega_u)C_1 + C$. This concludes the proof.
\end{proof}

\begin{proof}[Proof of Theorem \ref{t:multi} (the one-phase boundary points)] We are now in position to conclude the regularity of the free boundary $\partial\Omega_i$ in a neighborhood of any one-phase boundary point $x_0\in\partial \Omega_i\cap\partial\mathcal D$. Indeed, we may assume that $x_0=0$ and that $\partial \mathcal D$ is the graph of a function $g$. Reasoning as above, we have that $\tilde u_i(x_1,x_2)=u_i(x_1, x_2+g(x_1))$ satisfies the almost-minimality condition from Lemma \ref{l:almost_almost_tilde_u} in a neighborhood of the origin. On the other hand, it is immediate to check that $\tilde u_i$ is still Lipschitz continuous. Thus, we can apply Theorem \ref{t:constrained} obtaining that, in a neighborhood of zero, $\partial\Omega_i$ is the graph of a $C^{1,\alpha}$ function.   
\end{proof}

\subsection{Two-phase points}\label{sub:two-phase}
Let $\Omega_i$ and $\Omega_j$ be two different sets from the optimal $n$-uple $(\Omega_1,\dots,\Omega_n)$, solution of \eqref{e:multi}. Let $u_i$ and $u_j$ be the first normalized eigenfunctions, respectively on $\Omega_i$ and $\Omega_j$. Finally, let $x_0\in\partial \Omega_i\cap\partial\Omega_j$. We know that there is a neighborhood $\mathcal U\subset\mathcal D$ of $x_0$ such that $\mathcal U\cap\Omega_k=\emptyset $, for every $k\notin\{i,j\}$. Setting $D:=\Omega_i\cup\Omega_j\cup \mathcal U$, we get that the function $u:=u_i-u_j$ is the solution of the two-phase problem
\begin{equation}\label{eq min part u}
\min \Big\{ \int_D |\nabla v|^2\,dx + q_i|\Omega_v^+| + q_j|\Omega_v^-| \ : \ v \in H^1_0(D), \ \int_D v_+^2\,dx=\int_D v_-^2\,dx = 1 \Big\}.
\end{equation}
We next show that the solutions of \eqref{eq min part u} satisfy a almost-minimality condition.
\begin{lemma}\label{l:almost_almost_2}
Let $D\subset\R^d$, $u\in H^1_0(D)$ be a Lipschitz continuous function on $\R^d$ and $L$ its Lipschitz constant. Suppose that $u$ is a solution of the minimization problem \eqref{eq min part u}. Then, there is some $r_0>0$ such that $u$ satisfies the following almost-minimality condition:

\noindent For every $r\in(0,r_0)$ and $x_0\in\partial\Omega_u$, 
\begin{align*}\label{e:almost_almost_uTP}
\int_{B_r(x_0)}& |\nabla u|^2\,dx + q_i |\Omega_u^+\cap B_r(x_0)|+q_j |\Omega_u^-\cap B_r(x_0)| \\
&\le \big(1+C_1r^{d+2}\big)\int_{B_r(x_0)} |\nabla v|^2\,dx +q_i |\Omega_v^+\cap B_r(x_0)|+q_j |\Omega_v^-\cap B_r(x_0)|+C_2r^{d+2},\notag
\end{align*} 
for every $v \in H^1_0(D)$ such that $u-v\in H^1_0(B_r(x_0))$, where $C_1=2L^2$ and $\displaystyle C_2=C_1\int_D|\nabla u|^2\,dx$.
\end{lemma}

\begin{proof}
Follows precisely as in Lemma \ref{l:almost_almost_u}.
\end{proof}
We are now in position to complete the proof of  Theorem \ref{t:multi}.
\begin{proof}[Proof of Theorem \ref{t:multi} (the two-phase free boundary)] 
We only need to notice that in a neighborhood of any two-phase point $x_0\cap\partial\Omega_i\cap\partial\Omega_j\cap \mathcal D$, Lemma \ref{l:almost_almost_2} implies that $u$ is a almost-minimizer of $J_{\text{\tiny\sc tp}}$, where the matrix $A$ is the identity, $Q_+=q_i$ and $Q_-=q_j$. Thus, it is sufficient to apply Theorem \ref{t:two-phase}.
\end{proof}

\appendix
\section{The flat one-phase free boundaries are $C^{1,\alpha}$}
In this section we discuss a regularity theorem for viscosity solutions of the one-phase problem (without constraint). We fix $f:B_2\to\R$ to  be a bounded and continuous function and $A:B_2\to Sym_d^+$ to be a matrix-valued coercive and bounded function with H\"older continuous coefficients, as in the Introduction. Before we state the result, we recall that the continuous function $u:\R^d\supset B_1\to\R$, $u\ge 0$ is a viscosity solution of 
\begin{equation}\label{e:viscosity2}
-\text{div}(A\nabla u)=f\quad\text{in}\quad \Omega_u\cap B_1,\qquad \big|A^{\sfrac12}[\nabla u]\big|=g\quad\text{on}\quad \partial\Omega_u\cap B_1,
\end{equation}
if the first equation holds in the open set $\Omega_u$ and if, for every $x_0\in\partial\Omega_u$ and every $\varphi\in C^{\infty}(\R^d)$ touching $u\circ F_{x_0}$ from above (below) at zero, we have that $|\nabla \varphi|(0)\ge g(x_0)$ (resp. $|\nabla \varphi|(0)\le g(x_0)$). Recall that {\it touching from above (below)} means that $\varphi(0)=0$ and $\varphi\ge u\circ F_{x_0}$ (resp. $\varphi\le u\circ F_{x_0}$) in $\Omega_u\cap B_1$. Moreover, we suppose that $g$ is H\"older continuous and that there are constants $\eta_g>0$, $C_g>0$ and $\delta_g>0$ such that 
\begin{equation}\label{e:holder:g}
\begin{cases}
\begin{array}{ll}
|g(x)-g(y)|\le C_g|x-y|^{\delta_g}\quad\text{for every}\quad x,y\in\partial\Omega_u\cap B_1,\\
\eta_g\le g(x)\quad\text{for every}\quad x\in\partial\Omega_u\cap B_1.	
\end{array}
\end{cases}
\end{equation}
The following result follows immediately from the results proved in \cite{desilva}.
\begin{theorem}[Flat free boundaries are $C^{1,\alpha}$]\label{t:desilva}
Suppose that $u:B_1\to \R$ is a viscosity solution of \eqref{e:viscosity2} and that $g:\partial\Omega_u\to \R$ satisfies \eqref{e:holder:g}. Then, there exist $\eps>0$ and $\rho>0$ such that if $x_0\in \partial\Omega_u\cap B_1$ and $u$ is such that
$$g(x_0)\,\max\{0, x\cdot \nu -\eps\rho\}\le u\circ F_{x_0}(x)\le g(x_0)\,\max\{0,x\cdot \nu +\eps\rho\}\quad\text{for every}\quad x\in B_\rho,$$
then $\partial\Omega_u$ is $C^{1,\alpha}$ in $B_{\sfrac\rho2}(x_0)$. 
\end{theorem}
\begin{remark}
Notice that since in dimension two all the blow-up limits of $u_+$ (given by Theorem \ref{t:two-phase}) are half-plane solutions (Proposition \ref{p:classification}), we have that the flatness assumption of the above Theorem is satisfied at every point of the free boundary $\partial\Omega_u^+$. We also notice that, in our case, we have $g=\mu_+$, which is H\"older continuous by Lemma \ref{l:holder}.
\end{remark}

\begin{definition}[Flatness]
 Let $u:B_1\to\R$ be continuous, $u\ge 0$ and $u\in H^1(B_1)$. 	We say that $u$ is $(\eps,\nu)$-flat, if there are a matrix-valued $A:B_1\to Sym_d^+$ with H\"older continuous coefficients, and a continuous $f:B_1\to\R$ such that: 
\begin{align}
	-\text{\rm div}\,(A\nabla u)=f\quad\text{in}\quad \Omega_u\cap B_1\,;\label{e:flatness:equation}\\
	 \|f\|_{L^\infty(B_1)}\le\eps\quad\text{and}\quad \|a_{ij}-\delta_{ij}\|_{L^\infty(B_1)}\le \eps^2\quad\text{for every}\quad 1\le i,j\le d\,;\label{e:flatness:Af}\\
	1-\eps^2\le |\nabla u|\le 1+\eps^2\quad\text{on}\quad \partial\Omega_u\cap B_1\,;\label{e:flatness:gradient}\\
	\max\{0, x\cdot \nu -\eps\}\le u(x)\le \max\{0, x\cdot \nu +\eps\}\quad\text{for every}\quad x\in B_1\label{e:flatness:geometric}.
\end{align}
\begin{remark}
The condition \eqref{e:flatness:gradient} is intended in a viscosity sense, that is, for any $\varphi\in C^{\infty}(B_1)$, we have: 
\begin{itemize}
\item if $\varphi(x_0)=u(x_0)$ for some $x_0\in \partial \Omega_u\cap B_1$ and $\varphi^+\ge u$ in $\Omega_u\cap B_1$, then $|\nabla \varphi(x_0)|\ge 1-\eps^2$;
\item if $\varphi(x_0)=u(x_0)$ for some $x_0\in \partial \Omega_u\cap B_1$ and $\varphi\le u$ in $\Omega_u\cap B_1$, then $|\nabla \varphi(x_0)|\le 1+\eps^2$.
	\end{itemize}
\end{remark}
\end{definition}
In order to prove Theorem \ref{t:desilva} one has to show that the flatness improves at lower scales, that is, if $u$ is $(\eps,\nu)$-flat, then a rescaling $u_r$ of $u$ is $(\sfrac\eps2,\nu')$-flat for some $\nu'$, which is close to $\nu$. Of course, the essential (and hardest) part of the proof is to show the improvement of the geometric flatness \eqref{e:flatness:geometric}. This was proved by De Silva in \cite[Lemma 4.1]{desilva}.
\begin{lemma}[Improvement of the geometric flatness]\label{l:geometric}
There are universal constants $C>0$, $r_0>0$ and $\eps_0>0$ such that if $u$ is $\eps$-flat in the direction $\nu$, for some $\eps\in(0,\eps_0)$ and $\nu\in \partial B_1$, then, for every $r\in(0,r_0)$ there is some $\nu'\in\partial B_1$ such that $|\nu-\nu'|\le C\eps^2$ and 
$$\max\left\{0, x\cdot \nu' -\frac\eps2\right\}\le u_r(x)\le \max\left\{0, x\cdot \nu' +\frac\eps2\right\}\quad\text{for every}\quad x\in B_1,$$
where $u_r:B_1\to\R$ is the one-homogeneous rescaling $u_r(x)=\frac{u(rx)}r$.
\end{lemma}	
\begin{proof}[\bf Proof of Theorem \ref{t:desilva}] We will first prove that the flatness condition \eqref{e:flatness:equation}-\eqref{e:flatness:geometric} improves at smaller scales. We fix $x_0\in \partial\Omega_u\cap B_1$ and we consider the function $\tilde u=\frac{1}{g(x_0)}u\circ F_{x_0}$ (recall that $F_{x_0}(x)=x_0+A_{x_0}^{\sfrac12}[x]$). Let $\eps$ and $r_0$ be the constants from Lemma \ref{l:geometric}. We will prove that there is $r_1\le r_0$ such that: if $\tilde u$ is $(\eps,\nu)$-flat, then for every $r\le r_1$, $\tilde u_r$ is $(\sfrac{\eps}2,\nu')$-flat, for $\nu'$ given again by Lemma \ref{l:geometric}. It is sufficient that the conditions \eqref{e:flatness:equation}, \eqref{e:flatness:Af} and \eqref{e:flatness:gradient} are satisfied for $\tilde u_r$ with the flatness parameter $\sfrac{\eps}2$.
We notice that $\tilde u$ is a viscosity solution of 
\begin{equation}\label{e:viscosity22}
-\text{div}(\tilde A\nabla \tilde u)=\tilde f\quad\text{in}\quad \Omega_{\tilde u},\qquad \big| \tilde A^{\sfrac12}[\nabla \tilde u]\big|=\tilde g\quad\text{on}\quad \partial\Omega_{\tilde u},
\end{equation}
where $\tilde A_x=A_{x_0}^{-\sfrac12}A_{F_{x_0}(x)}A_{x_0}^{-\sfrac12}$, $\tilde f=\frac1{g(x_0)}f\circ F_{x_0}$, $\tilde g=\frac{1}{g(x_0)}g\circ F_{x_0}$ and $\tilde A_x^{\sfrac12}=A_{F_{x_0}(x)}^{\sfrac12}\circ A_{x_0}^{-\sfrac12}$. Notice that $0\in\partial\Omega_{\tilde u}$ and set $\displaystyle\tilde u_r(x):=\frac{\tilde u(rx)}{r}$. Thus, for small enough $r>0$, $\tilde u_r$ is a viscosity solution of 
\begin{equation}\label{e:viscosity33}
-\text{div}(\tilde A_r\nabla \tilde u_r)=\tilde f_r\quad\text{in}\quad \Omega_{\tilde u}\cap B_1,\qquad \big| \tilde A_r^{\sfrac12}[\nabla \tilde u_r]\big|=\tilde g_r\quad\text{on}\quad \partial\Omega_{\tilde u}\cap B_1,
\end{equation}
where $\tilde A_r(x):=\tilde A(rx)$, $\tilde f_r(x)=r\tilde f(rx)$, $\tilde g_r(x)=\tilde g(rx)$ and $\tilde A_r^{\sfrac12}(x)=A_{F_{x_0}(rx)}^{\sfrac12}\circ A_{x_0}^{-\sfrac12}$. Now, if $\tilde u$ is $(\eps,\nu)$-flat, then the H\"older continuity of the coefficients $a_{ij}$ and the boundedness of $f$ imply that  \eqref{e:flatness:Af} holds with $\sfrac\eps2$ and $\tilde u_r$, for any $r\le r_1$, where $r_1\le r_0$, is a universal constant depending on the H\"older norm of $a_{ij}$. Now, in order to get \eqref{e:flatness:gradient} for $\sfrac\eps2$ and $\tilde u_r$, we suppose that $\varphi\in C^{\infty}(B_1)$ touches $\tilde u_r$ from below at a point $y_0\in B_1\cap \partial\{\tilde u_r>0\}$. Thus, we have that 
$$\big|A_{F_{x_0}(ry_0)}^{\sfrac12}\circ A_{x_0}^{-\sfrac12}[\nabla \varphi(y_0)]\big|\le \frac{g(F_{x_0}(ry_0))}{g(x_0)},$$
and so, if $\|\cdot\|=\|\cdot\|_{\mathcal L(\R^d)}$ stands for the space of $d\times d$ matrices, we have 
$$|\nabla \varphi(y_0)|\le \| A_{x_0}^{\sfrac12}\circ A_{F_{x_0}(ry_0)}^{-\sfrac12}\|\frac{g(F_{x_0}(ry_0))}{g(F_{x_0}(0))}.$$
Now, by the H\"older continuity (and the uniform boundedness from below) of $g$, we can choose $r_1$ such that 
$$\frac{g(F_{x_0}(ry_0))}{g(F_{x_0}(0))}\le 1+\frac{\eps^2}{10}.$$
On the other hand, there are universal constants $C$ and $\delta>0$, depending only on the H\"older exponent $\delta_A$ and the norm $C_A$, of the matrix-valued function $A$, such that 
$$\| A_{x_0}^{\sfrac12}\circ A_{F_{x_0}(ry_0)}^{-\sfrac12}-Id\|\le \| A_{x_0}^{\sfrac12}- A_{F_{x_0}(ry_0)}^{\sfrac12}\|\|A_{F_{x_0}(ry_0)}^{-\sfrac12}\|\le C |ry_0|^{\delta}\le C r_1^{\delta}.$$
Choosing $r_1$ such that $C r_1^{\delta}\le \frac{\eps^2}{10}$ and using the triangular inequality, we get 
$$|\nabla \varphi(y_0)|\le \| A_{x_0}^{\sfrac12}\circ A_{F_{x_0}(ry_0)}^{-\sfrac12}\|\frac{g(F_{x_0}(ry_0))}{g(F_{x_0}(0))}\le \left(1+\frac{\eps^2}{10}\right)^2\le 1+(\sfrac{\eps}2)^2,$$
which completes the proof of the improvement of flatness for $\tilde u$, the case when $\varphi$ touches from above being analogous. Now, the claim follows by a standard argument, similar to the one we used in Section \ref{s:constrained}.
\end{proof}

\bigskip\bigskip
\noindent {\bf Acknowledgments.} 
The first author has been partially supported by the NSF grant DMS 1810645. The second and third author have been partially supported by Agence Nationale de la Recherche (ANR) by the projects GeoSpec (LabEx PERSYVAL-Lab, ANR-11-LABX-0025-01) and CoMeDiC (ANR-15-CE40-0006).

\end{document}